\documentclass[12pt,reqno]{amsproc}
\usepackage[english]{babel}
\usepackage{ucs}
\usepackage{amssymb,amsmath}
\usepackage{indentfirst}
\usepackage[left=3.2cm,right=3.2cm,top=3cm,bottom=3cm,bindingoffset=0cm]{geometry}
\usepackage{amsthm}
\usepackage{amsfonts}

\numberwithin{equation}{section}
\theoremstyle{plain}

\newtheorem{theorem}{Theorem}
\newtheorem{lemma}{Lemma}
\newtheorem{corollary}{Corollary}

\newtheorem{oldassertion}{Theorem}

\newtheorem{oldlemma}{Lemma}

\newtheorem{remark}{Remark}

\renewcommand{\Im}{\mathop{\mathrm{Im}}\nolimits}

\def\Im{\operatorname{Im}}

\usepackage{xcolor}
\usepackage{hyperref}

\hypersetup{pdfstartview=FitH, linkcolor=linkcolor,urlcolor=urlcolor, colorlinks=true, linkcolor=blue, citecolor=blue}
\begin{document}

\title{UNIFORM CONVERGENCE CRITERION FOR NON-HARMONIC SINE SERIES}
\author{Kristina Oganesyan}
\address{K. Oganesyan, Lomonosov Moscow State University, Moscow Center for Fundamental and Applied Mathematics, Universitat Aut\`onoma de Barcelona, Centre de Recerca Matem\`atica}
\email{oganchris@gmail.com}

\date{}

\maketitle

\markright{Uniform convergence criterion for non-harmonic sine series}

\begin{abstract}
We show that for a nonnegative monotone sequence $\{c_k\}$ the condition $c_kk\to 0$ is sufficient for uniform convergence of the series $\sum_{k=1}^{\infty}c_k\sin k^{\alpha} x$ on any bounded set for $\alpha\in (0,2)$, and for an odd natural $\alpha$ it is sufficient for uniform convergence on the whole $\mathbb{R}$. Moreover, the latter assertion still holds if we replace $k^{\alpha}$ by any polynomial in odd powers with rational coefficients. On the other hand, in the case of an even $\alpha$ it is necessary that $\sum_{k=1}^{\infty}c_k<\infty$ for convergence of the mentioned series at the point $\pi/2$ or at the point $2\pi/3$. Consequently, we obtain uniform convergence criteria. Besides, the results for a natural $\alpha$  remain true for sequences from more general RBVS class. 

\end{abstract}

{\bf Key words.} Uniform convergence, sine series, monotone coefficients, fractional parts of the values of a polynomial, Weyl sums.

\section{Introduction}
We consider the series
\begin{align}\label{series}
\sum_{k=1}^{\infty}c_k\sin k^{\alpha} x, \quad c_k \searrow 0,
\end{align}
for $\alpha>0$, and, for odd natural $\alpha$, the more general ones
\begin{align}\label{series2}
\sum_{k=1}^{\infty}c_k\sin f(k) x, \quad c_k\searrow 0,
\end{align}
where $f(k)$ stands for a polynomial of degree $\alpha$ with rational coefficients in odd powers of $k$. In the case of a natural $\alpha$ we also consider sequences $\{c_k\}$ from a more general class. We are interested in conditions which would be necessary and sufficient for uniform convergence of series \eqref{series} and \eqref{series2}.

For the case $\alpha=1$ such conditions are well-known (see \cite{CJ}).
\begin{oldassertion}\label{A} {\it (Chaundy, Jolliffe, 1916) If a nonnegative sequence $\{c_k\}_{k=1}^{\infty}$ is nonincreasing, then the series $\sum_{k=1}^{\infty}c_k\sin kx$ converges uniformly on $\mathbb{R}$ if and only if $c_kk\to 0$ as $k\to\infty$.}
\end{oldassertion}

In \cite{N}, the requirement of monotonicity is released to the requirement of quasimonotonicity, that is, of existence of a nonnegative number $\gamma$ such that $c_k k^{-\gamma}$ decrease, and the same criterion was extended to some more general sequences in \cite{S}. One more generalization of Theorem \ref{A} we can find in \cite{L}, where the corresponding criterion was proved for the sequences from RBVS class, i.e., satisfying the following conditions
\begin{align}\label{rbvs}
\sum_{k=l}^{\infty}|c_k-c_{k+1}|\leq Vc_l,\qquad c_kk\to 0\;\text{ as}\;k\to\infty,
\end{align} 
for any $l$, where $V$ depends only on $\{c_k\}$.

For a more general class of sequences, containing all the classes mentioned above, the result was obtained in \cite{T}.
\begin{oldassertion} {\it (Tikhonov, 2007) If a nonnegative sequence $\{c_k\}_{k=1}^{\infty}$ belongs to GM class, i.e., if there exists a constant $A$ depending only on $\{c_k\}$ such that}
$$\sum_{k=l}^{2l-1}|c_k-c_{k+1}|\leq A c_l$$
{\it for all $l$, then the series $\sum_{k=1}^{\infty}c_k\sin kx$ converges uniformly on $\mathbb{R}$ if and only if $c_kk\to 0$ as $k\to\infty$.}
\end{oldassertion}
Moreover, there are uniform convergence criteria for the series $\sum_{k=1}^{\infty}c_k\sin kx$ with coefficients satisfying various conditions of general monotonicity (see \cite{T2} and \cite{DMT}).

The cases $\alpha=\frac{1}{2}$ and $\alpha=2$ for series \eqref{series} were considered in \cite{K}, where it was shown that the condition $c_kk\to 0$ is necessary and sufficient for uniform convergence of series \eqref{series} on the interval $[0, \pi]$ for $\alpha=\frac{1}{2}$, and for $\alpha=2$ a necessary and sufficient condition is $\sum_{k=1}^{\infty}c_k<\infty$.

We obtain the following

\begin{theorem}\label{thm} Let a nonnegative sequence $\{c_k\}_{k=1}^{\infty}$ be nonincreasing. Then 

(a)\quad if $\alpha$ is an even natural number, then series \eqref{series} converges at the point $\frac{\pi}{2}$ or at the point $\frac{2\pi}{3}$ only if $\;\sum_{k=1}^{\infty}c_k<\infty$;

(b)\quad if $\alpha$ is an odd natural number, then for uniform convergence of series \eqref{series2} on $\mathbb{R}$ it is sufficient that $c_kk\to 0$ as $k\to\infty$;

(c)\quad if $\alpha\in (0,2)$, then for uniform convergence of series \eqref{series} on any bounded subset of $\mathbb{R}$ it is sufficient that $c_kk\to 0$ as $k\to\infty$.
\end{theorem}

\begin{remark} In particular, it follows from Theorem \ref{thm} that for an odd $\alpha$ the sum of the series $$\sum_{k=1}^{\infty}\frac{a_k\sin k^{\alpha}x}{k},\quad\frac{a_k}{k}\searrow 0,\; a_k\to 0,$$ represents a continuous function, as long as the function $\sum_{k=1}^{\infty}\frac{\sin k^{\alpha}x}{k}$, despite being bounded, is discontinuous on a set which is dense in $\mathbb{R}$. More precisely, it has discontinuities at all points of the form $2\pi a/b,\;a,b\in\mathbb{Z},$ such that $\sum_{k=1}^{b}e^{ \frac{2\pi k^{\alpha}a}{b}}\neq 0$ (see \cite[Sec. 3]{Os}). Meanwhile, it is known that for any natural $a,\;n>2$ and any prime $p>n$ such that $(a,p)=1$ there holds 
$$\sum_{k=1}^{p^n}e^{\frac{2\pi a k^n}{p^n}}=p^{n-1}$$
(see \cite[(72)]{korobov}), and the set of $\pi$-rational points of the form $2\pi a/p^n,\;(a,p)=1,$ for a fixed $n$, is dense in $\mathbb{R}$.
\end{remark}

Theorem \ref{thm} represents the essential part of the uniform convergence criterion of series \eqref{series}, which is to be formulated later (in Theorem \ref{thm2}).

\begin{remark} 
For the part (a), one can find some other points with the same property, but it does not seem essential. 
\end{remark}

\begin{remark} If instead of sine series \eqref{series} we consider the corresponding cosine series, we can easily notice that the condition $\sum_{k=1}^{\infty}c_k<\infty$ is necessary for its convergence at the point $0$, and for a natural $\alpha$ --- for convergence at the points of the form $2\pi m, \;m\in\mathbb{Z}$.
\end{remark}

In the proof of Theorem \ref{thm} for the case \eqref{series2} we deal with distributions of the fractional parts of the values of a polynomial (see, for example, \cite{V},\cite{P}) and with estimates of Weyl sums (see also \cite{CHD},\cite{HB} and \cite{W}), which play an important role in number theory, in particular, in solving Waring's problem of representation of a natural number as a sum of equal powers of natural numbers and in estimating sums appearing in the Riemann zeta-function theory. The following well-known theorems provide bounds of Weyl sums at points of a special kind.

\begin{oldassertion}{\it (Weyl, 1916, \cite[Th. 14]{korobov}) Let $n\geq 2,\;h(x)=\alpha_1x+...+\alpha_nx^n$ and}
$$\alpha_n=\frac{a}{q}+\frac{\theta}{q^2},\quad (a,q)=1,\quad |\theta|\leq 1.$$
{\it If $0<\varepsilon_1<1$ and $P^{\varepsilon_1}\leq q\leq P^{n-\varepsilon_1}$, then for any $0<\varepsilon<1$ there holds}
$$\bigg|\sum_{k=1}^Pe^{2\pi i h(k)}\bigg|\leq C(n,\varepsilon,\varepsilon_1)P^{1-\frac{\varepsilon_1-\varepsilon}{2^{n-1}}}.$$
\end{oldassertion}

\begin{oldassertion}{\it (Vinogradov, 1952, \cite[Th. 17]{korobov}) Let $n>2,\;h(x)=\alpha_1x+...+\alpha_nx^n$ and }
$$\alpha_n=\frac{a}{q}+\frac{\theta}{q^2},\quad (a,q)=1,\quad |\theta|\leq 1.$$
{\it If $P\leq q\leq P^{n-1}$, then there holds}
$$\bigg|\sum_{k=1}^Pe^{2\pi i h(k)}\bigg|\leq e^{3n}P^{1-\frac{1}{9n^2\ln n}}.$$
\end{oldassertion}

However, in these theorems the length of the sum $P$ is tied to the denominators of rational appoximations of the leading coefficient of the polynomial $h$.

We obtain Weyl sums estimates depending on how well the leading coefficient of the polynomial is approximated by rationals with denominators less than some small power of $P$. The best estimates are obtained at the points which are approximated in this way rather badly. Such estimates are of interest because if the leading coefficient is $``$close$"$ to a rational, then the Weyl sums behave in some sense similar to rational sums which are well studied and easier to deal with.

From the proof of Theorem \ref{thm} (b) it follows that when $c_kk\to 0$, the series $$\sum_{m=1}^{\infty}(c_m-c_{m+1})\Big|\Im\sum_{k=0}^{m} e^{if(k)x}\Big|$$ converges uniformly. This means, in particular, that if we consider the coefficients $c_m:=m^{-1}\ln^{-1}(m+1)$, the series $\sum_{m=1}^{\infty}\frac{1}{m^2\ln (m+1)}\left|\Im\sum_{k=0}^{m} e^{if(k)x}\right|$ converges uniformly, hence, for any $a>0$, the number of $m$ such that $|\Im\sum_{k=0}^{m} e^{if(k)x}|\geq am$ is uniformly small. 

It is worth mentioning that in \cite{O} an estimate of symmetric partial sums of the series $\sum_{k\in\mathbb{Z}\setminus \{0\}}\frac{e^{2\pi i h(k)}}{k}$ is given for a polynomial $h$ with real coefficients. This result is used to establish lower estimates for Lebesgue constants and prove the following theorem which is strongly related to the present work.

\begin{oldassertion}{\it (Oskolkov, 1986)
Let $r\geq 2,\;P_r(y)=\alpha_0+\alpha_1 y+...+\alpha_r y^r$ be a polynomial with integer coefficients assuming different integer values for $y\in\mathbb{N}\cup\{0\}$. Then $\{P_r(n)\}$ is not a spectrum of uniform convergence.}
\end{oldassertion}

Here by a spectrum of uniform convergence we mean a sequence $\mathfrak{K}=\{k_n\}$ of pairwise different integers such that for any continuous function, having its coefficients equal zero for $k\notin\mathfrak{K}$, partial sums of its Fourier series converge uniformly.

In \cite{AO}, uniform boundness of the symmetric partial sums $\sum_{1\leq |k|\leq m}\frac{e^{2\pi i h(k)}}{k}$ in $m\in\mathbb{N}$ and $\deg h\leq r$, for a fixed $r$, was proved. In particular, this result leads to

\begin{oldassertion}\label{AOs}{\it (Arkhipov, Oskolkov, 1987)
Let $P^+(x),\;P^-(x)$ be polynomials with real coefficients and $P^+(-x)\equiv P^+(x),\; P^-(-x)\equiv-P^-(x).$ Then the series}
$$\sum_{n=1}^{\infty}\frac{e^{2\pi iP^+(n)}\sin 2\pi P^-(n)}{n}$$
{\it converges and the absolute values of its partial sums are bounded by a constant depending only on the powers of $P^+$ and $P^-$ but not on their coefficients.}
\end{oldassertion}
\vspace{4pt}

Using the Abel transformation from Theorem \ref{AOs} one can derive Theorem \ref{thm} (b) in the case $c_kk\searrow 0$ .

In order to formulate the uniform convergence criterion for series \eqref{series} we need to introduce the following definition.

For $\alpha>0$ and $\gamma>0$ we call by {\it a discrete $(\alpha,\gamma)$-neighbourhood of zero} such a sequence $\{x_j\}_{j=0}^{\infty}$ that $|x_j|=\frac{\pi}{{\gamma}^{\alpha+1}(N+j)^{\alpha}}$ for all $j\in\mathbb{Z}^+$ and some $N\in\mathbb{N}$.

Now we are ready to formulate the criterion.
\begin{theorem}\label{thm2} Let a nonnegative sequence $\{c_k\}_{k=1}^{\infty}$ be nonincreasing. Then 

(a)\quad if $\alpha$ is an even natural number, then series \eqref{series} converges uniformly on a set containing a point of the form $\frac{\pi}{2}+2\pi m$ or $\frac{2\pi}{3}+2\pi m,\;m\in \mathbb{Z},$ if and only if $\sum_{k=1}^{\infty}c_k<\infty$;

(b)\quad if $\alpha$ is an odd natural number, then series \eqref{series} converges uniformly on a set containing for some $\gamma\geq 2$ a discrete $(\alpha,\gamma)$-neighbourhood of zero if and only if $c_kk\to 0$ as $k\to\infty$;

(c)\quad if $\alpha\in (0,2)$, then series \eqref{series} converges uniformly on a bounded set containing for some $\gamma\geq 2$ a discrete $(\alpha,\gamma)$-neighbourhood of zero if and only if $c_kk\to 0$ as $k\to\infty$.
\end{theorem}

\begin{remark}\label{rbvs_rem} The parts (a) and (b) of Theorems \ref{thm} and \ref{thm2} remain true if we replace the condition of monotonicity of the coefficients $\{c_k\}$ by belonging to RBVS class (see \eqref{rbvs}).
\end{remark} 

\begin{remark} In particular, in the part (b) (in the part (c)) the condition $c_k k\to 0$ is necessary and sufficient for uniform convergence of series \eqref{series} on any (bounded) set containing a punctured neighbourhood of zero. 
\end{remark}

It will be easy to see that the criterion could be slightly generalized by means of adding some extra parameters in the definition of a discrete $(\alpha,\gamma)$-neighbourhood of zero. We will not do that to avoid the formulation of Theorem \ref{thm2} being excessively tedious.

\section{Weyl sums estimates depending on rational approximations of the leading coefficient of the polynomial}\label{Weyl}

\begin{lemma}\label{lem1} Let $P\in\mathbb{N},\;1\leq A\in\mathbb{R}$. Then for any natural $k\geq 1$ there holds
$$\#\left\{(y_1,y_2,...,y_k)\in \{1,2,...,P\}^k: y_1y_2...y_k\leq\frac{P^k}{A}\right\}\leq \frac{kP^k}{A^{1/k}}.$$
\end{lemma}

\begin{proof}
The assertion follows from the successive inequalities
\begin{align*}
\#\bigg\{(y_1,y_2,...,y_k)\in &\{1,2,...,P\}^k: y_1y_2...y_k\leq\frac{P^k}{A}\bigg\}\\
\leq k \;\cdot\;\#\bigg\{(y_1,&y_2,...,y_k)\in \{1,2,...,P\}^k: y_1\leq y_2,...,y_k, \;y_1y_2...y_k\leq\frac{P^k}{A}\bigg\}\\
\leq k \;\cdot\;\#\bigg\{(&y_1,y_2,...,y_k)\in \{1,2,...,P\}^k: 1\leq y_1\leq \frac{P}{A^{1/k}}\bigg\}=\frac{kP^k}{A^{1/k}}.
\end{align*}
\end{proof}

We now formulate a statement \cite[L. 13]{korobov} which will be used on several occasions.

\begin{oldlemma} Let $\lambda$ and $x_1,...,x_k$ be natural numbers. Denote by $\tau_k(\lambda)$ the number of the solutions of the equation $x_1...x_k=\lambda$. Then for any $\varepsilon\in (0,1)$ the following estimate is fulfilled
\begin{align}\label{lemma13}
\tau_k(\lambda)\leq C_k(\varepsilon)\lambda^{\varepsilon},
\end{align}
where $C_k(\varepsilon)$ is a constant depending only on $k$ and $\varepsilon$.
\end{oldlemma}

For any number $y$, we denote
$$\|y\|:=\min\big(\{y\},1-\{y\}\big),$$
where $\{y\}$ stands for the fractional part of $y$.

Further, for any function $\psi(y)$ and number $y_1$, we denote by $$\underset{y_1}{\Delta}\psi(y)=\psi(y+y_1)-\psi(y)$$ the first order difference of the function $\psi(y)$, and for $k\geq 2$ we define the $k$-th order difference inductively
$$\underset{y_1,...,y_{k}}{\Delta}\psi(y)=\underset{y_k}{\Delta}\Big(\underset{y_1,...,y_{k-1}}{\Delta}\psi(y)\Big).$$

According to \cite[(144)]{korobov}, if $\psi(y)$ is a polynomial of degree $k\geq 2$, then 
\begin{align}\label{Delta}
\underset{y_1,...,y_{k-1}}{\Delta}\psi(y)=k! \alpha_k y_1...y_{k-1} y+\eta,
\end{align}
where $\alpha_k$ is the leading coefficient of $\psi(y)$, and $\eta$ depends only on the coefficients of $\psi(y)$ and on the numbers $y_1,...,y_{k-1}$. Also, due to \cite[L. 12]{korobov}, for any $K,k\geq 1$ there holds
\begin{align}\label{reference}
\bigg|\sum_{y=1}^K e^{2\pi i h(y)}\bigg|^{2^k}\leq 2^{2^k}K^{2^k-(k+1)}\sum_{y_1=0}^{K_1-1}...\sum_{y_k=0}^{K_k-1}\bigg|\sum_{y=1}^{K_{k+1}}e^{2\pi i \underset{y_1...y_{k}}{\Delta}h(y)}\bigg|,
\end{align}
where $K_1:=K,\;K_{\nu+1}:=K_{\nu}-y_{\nu},\;\nu=1,2,...,k$. Now, taking into account \eqref{Delta} and \eqref{reference}, for any polynomial $f$ of degree $n$ with the leading coefficient $\alpha_n$ we obtain
\begin{align}\label{ineq_book}
&\qquad\bigg|\sum_{k=1}^m e^{if(k) x}\bigg|^{2^{n-1}}\leq\\
&\leq 2^{2^{n-1}}m^{2^{n-1}-n}\sum_{y_1=0}^{m-1}\sum_{y_2=0}^{m-y_1-1}...\sum_{y_{n-1}=0}^{m-y_1-y_2-...-y_{n-2}-1}\bigg|\sum_{y=1}^{m-y_1-...-y_{n-1}}e^{i\underset{y_1...y_{n-1}}{\Delta}f(y)x}\bigg|\nonumber\\
&=2^{2^{n-1}}m^{2^{n-1}-n}\sum_{y_1=0}^{m-1}\sum_{y_2=0}^{m-y_1-1}...\sum_{y_{n-1}=0}^{m-y_1-y_2-...-y_{n-2}-1}\bigg|\sum_{y=1}^{m-y_1-...-y_{n-1}}e^{in!yy_1...y_{n-1}\alpha_n x}\bigg|\nonumber\\
&\leq 2^{2^{n-1}}m^{2^{n-1}-n}\Bigg((n-1)m^{n-1}+\sum_{y_1,...,y_{n-1}=1}^{m}\bigg|\sum_{y=1}^{m-y_1-...-y_{n-1}}e^{in!yy_1...y_{n-1}\alpha_n x}\bigg|\Bigg).\nonumber
\end{align}
Note that for any $t$ and any natural $l$ there holds 
\begin{align}\label{sin_e}
\bigg|\sum_{y=1}^l e^{iyt}\bigg|&\leq \bigg|\sum_{y=1}^l\sin yt\bigg|+\bigg|\sum_{y=1}^l \cos yt\bigg|=\bigg|\frac{\cos \frac{t}{2}-\cos\frac{(2l+1)t}{2}}{2\sin\frac{t}{2}}\bigg|\nonumber\\
+\bigg|\frac{\sin\frac{t}{2}-\sin\frac{(2l+1)t}{2}}{2\sin\frac{t}{2}}\bigg|&=\bigg|\frac{\sin \frac{lt}{2}\sin\frac{(l+1)t}{2}}{\sin\frac{t}{2}}\bigg|+\bigg|\frac{\sin\frac{lt}{2}\cos\frac{(l+1)t}{2}}{\sin\frac{t}{2}}\bigg|\leq 2\bigg|\frac{\sin\frac{lt}{2}}{\sin\frac{t}{2}}\bigg|.
\end{align}
Combining \eqref{ineq_book} and \eqref{sin_e}, we derive
\begin{align}\label{fin_S_m}
\bigg|\sum_{k=1}^m e^{if(k) x}\bigg|^{2^{n-1}}&\leq 2^{2^{n-1}}m^{2^{n-1}-1}(n-1)\qquad\qquad\qquad\qquad\qquad\\
+2^{2^{n-1}+1}m^{2^{n-1}-n}&\sum_{y_1,...,y_{n-1}=1}^{m}\bigg|\frac{\sin\frac{(m-y_1-...-y_{n-1})n!y_1...y_{n-1}\alpha_n x}{2}}{\sin\frac{n!y_1...y_{n-1}\alpha_n x}{2}}\bigg|\nonumber\\
&\leq 2^{2^{n-1}}m^{2^{n-1}-1}(n-1)\nonumber\\
+2^{2^{n-1}+1}m^{2^{n-1}-n}&\sum_{y_1,...,y_{n-1}=1}^{m}\bigg|\min\bigg\{m,\frac{1}{2\|\frac{n!y_1...y_{n-1}\alpha_n x}{2\pi}\|}\bigg\}\bigg|.\nonumber
\end{align}

\begin{lemma}\label{lm1} Let $0<y\in\mathbb{R}\setminus\mathbb{Q},\;\varepsilon\in(0,1),\;4\leq P\in\mathbb{N},\;3\leq n\in\mathbb{N}$. If there does not exist such a pair of coprime natural numbers $C$ and $M\leq P^{\varepsilon}$ that
$$\Big|y-\frac{C}{M}\Big|\leq P^{\varepsilon-1},$$
then there holds
\begin{align*}
\#\Big\{(y_1,y_2,...,y_{n-1})\in\{1,2,&...,P\}^{n-1}:\|yy_1...y_{n-1}\|\leq P^{\varepsilon-1}\Big\}\\
&\leq 4C_{n-1}\Big(\frac{\varepsilon}{2(n-1)}\Big)P^{n-1-\frac{\varepsilon}{2}},
\end{align*}
where $C_m(\gamma)$ is from \eqref{lemma13}. We also have
$$\sum_{y_1,...,y_{n-1}=1}^P\min\bigg\{P,\frac{1}{2\|yy_1...y_{n-1}\|}\bigg\}\leq GP^{n-\frac{\varepsilon}{2}},$$
where $G$ depends only on $n$ and $\varepsilon$.
\end{lemma}

\begin{remark} From now on we provide our arguments for irrational ($\pi$-irrational) numbers not because there is something realy different at rational ($\pi$-rational) ones but for the sake of simplicity.
\end{remark}

\begin{proof} Let $T$ be the minimal number from $\{1,2,...,P^{n-1}\}$ such that $\|yT\|\leq P^{\varepsilon-1}$ (if there is no such $T$, then the assertion becomes trivial). Then we have $T\geq P^{\varepsilon}$. In this case for any $0\leq k\leq P^{n-1}-T$ there is not more than one value among $\{y(k+1)\},\;\{y(k+2)\},...,\{y(k+T)\}$ which belongs to the half-interval $(0,P^{\varepsilon-1}]$ (otherwise for some $1\leq i<j\leq T$ there would hold $\|y(j-i)\|\leq P^{\varepsilon-1}$, which is imposible due to minimality of $T$) and not more than one value in the half-interval $[1-P^{\varepsilon-1},1)$. So, among the values $\{1,2,...,P^{n-1}\}$ there are not more than $2\lceil\frac{P^{n-1}}{T}\rceil\leq \frac{4P^{n-1}}{T}\leq 4P^{n-1-\varepsilon}$ values $k$ satisfying the condition $\|yk\|\leq P^{\varepsilon-1}$, and since for any $k$, according to \eqref{lemma13},

$$\#\Big\{(y_1,y_2,...,y_{n-1})\in\{1,2,...,P\}^{n-1}:y_1...y_{n-1}=k\Big\}\leq C_{n-1}\Big(\frac{\varepsilon}{2(n-1)}\Big)k^{\frac{\varepsilon}{2(n-1)}},$$
we have
\begin{align*}
&\#\Big\{(y_1,y_2,...,y_{n-1})\in\{1,2,...,P\}^{n-1}:\|yy_1...y_{n-1}\|\leq P^{\varepsilon-1}\Big\}\\
&\qquad\leq 4P^{n-1-\varepsilon}C_{n-1}\Big(\frac{\varepsilon}{2(n-1)}\Big)(P^{n-1})^{\frac{\varepsilon}{2(n-1)}}=4C_{n-1}\Big(\frac{\varepsilon}{2(n-1)}\Big)P^{n-1-\frac{\varepsilon}{2}}.
\end{align*}
Thus,
\begin{align*}
\sum_{y_1,...,y_{n-1}=1}^P\min\bigg\{P,\frac{1}{2\|yy_1...y_{n-1}\|}\bigg\}&\leq P\cdot 4C_{n-1}\Big(\frac{\varepsilon}{2(n-1)}\Big)P^{n-1-\frac{\varepsilon}{2}}\\
+P^{n-1}\frac{1}{2 P^{\varepsilon-1}}&\leq GP^{n-\frac{\varepsilon}{2}},
\end{align*}
where $G$ depends only on $n$ and $\varepsilon$.
\end{proof}

\begin{corollary}\label{cor1}
Under the conditions of Lemma \ref{lm1}, for any real monic polynomial $f$ of degree $n$, there holds 
$$\bigg|\sum_{k=1}^P e^{\frac{2\pi if(k)y}{n!}}\bigg|\leq D P^{1-\frac{\varepsilon}{2^n}},$$
where $D$ depends only on $n$ and $\varepsilon$.
\end{corollary}

\begin{proof} From \eqref{fin_S_m} and Lemma \ref{lm1} it follows that
\begin{align}\label{an}
\Big|\sum_{k=1}^P e^{\frac{2\pi if(k)y}{n!}}\Big|^{2^{n-1}}\leq 2^{2^{n-1}}P^{2^{n-1}-n}\big((n-1)P^{n-1}+2GP^{n-\frac{\varepsilon}{2}}\big)\leq D'P^{2^{n-1}-\frac{\varepsilon}{2}},
\end{align}
where $D'>0$ depends only on $\varepsilon$ and $n$. This leads to the needed result with $D=(D')^{\frac{1}{2^{n-1}}}$. 
\end{proof}

\begin{lemma}\label{lm2} Let $0<y\in\mathbb{R}\setminus\mathbb{Q},\;\varepsilon\in(0,\frac{1}{6}),\;9\leq P\in\mathbb{N},\;3\leq n\in\mathbb{N}$. If there exists such a pair of coprime natural numbers $C$ and $M\leq P^{\varepsilon}$ that
$$P^{\varepsilon-n}< \Big|y-\frac{C}{M}\Big|=:|\beta|\leq P^{\varepsilon-1},$$
then
\begin{align*}
&\qquad\qquad\#\Big\{(y_1,y_2,...,y_{n-1})\in\{1,2,...,P\}^{n-1}:\|yy_1...y_{n-1}\|\leq P^{\varepsilon-1}\Big\}\\
&\leq 6P^{n-\frac{3}{2}}C_{n-1}\Big(\frac{2\varepsilon}{n-1}\Big)+(n-1)C_{n-1}\Big(\frac{1}{2(n-1)}\Big)P^{n-1-\frac{n-\varepsilon}{n-1}}|\beta|^{-\frac{1}{n-1}}M^{-\frac{1}{2(n-1)}}
\end{align*}
and also
$$\sum_{y_1,...,y_{n-1}=1}^P\min\bigg\{P,\frac{1}{2\|yy_1...y_{n-1}\|}\bigg\}\leq U\big(P^{n-\varepsilon}+P^{n-\frac{n-\varepsilon}{n-1}}|\beta|^{-\frac{1}{n-1}}M^{-\frac{1}{2(n-1)}}\big),$$
where $U$ depends only on $n$ and $\varepsilon$.
\end{lemma}

\begin{proof} Without loss of generality assume that $P^{\varepsilon-n}<y-\frac{C}{M}=\beta\leq P^{\varepsilon-1}$. Suppose that there are coprime $C'$ and $M'$ which are distinct from $C$ and $M$ and satisfy the inequality
$$\Big|y-\frac{C'}{M'}\Big|=|\beta'|\leq\frac{2}{M'P^{1-\varepsilon}}.$$
Firstly, we have
$$\frac{1}{MM'}\leq \Big|y-\frac{C}{M}\Big|+\Big|y-\frac{C'}{M'}\Big|\leq \frac{3}{P^{1-\varepsilon}},$$
and hence,
\begin{align}\label{todo1}
M'\geq \frac{P^{1-\varepsilon}}{3M}\geq \frac{P^{1-2\varepsilon}}{3}\geq P^{\varepsilon}\geq M.
\end{align}
Secondly,
$$yMM'=C'M+\beta'M'M=CM'+\beta M'M,$$
so, $\{\beta'M'M\}=\{\beta M'M\}$. 
Thus, if $\beta'>0$, then since $\beta'M'\leq 2P^{\varepsilon-1},$ we have $\beta'M'M\leq 2P^{2\varepsilon-1}<1$ and $\{\beta' M'M\}=\beta' M'M$, hence, either $M'\geq \beta^{-1}M^{-1}$ or $\{\beta'M'M\}=\beta'M'M=\beta M'M$, from which $\beta'M'\geq\beta'M=\beta M$.

If $\beta'<0$, then $(-\beta'+\beta)M'M\geq 1$, which implies
\begin{align}\label{todo2}
M'\geq\frac{1+\beta'M'M}{\beta M}\geq (1-2P^{2\varepsilon-1})\beta^{-1}M^{-1}\geq \frac{1}{2}\beta^{-1}M^{-1}.
\end{align}
So we obtained that independently of the sign of $\beta'$ there holds either $\beta'M'\geq\beta M$ or $M'\geq \frac{1}{2}\beta^{-1}M^{-1}.$

Now, let $T_1<T_2<...<T_K$ be all the numbers $k$ from $\{1,2,...,P^{n-1}\}$ such that $\|yk\|\leq P^{\varepsilon-1}$. Then, since $\{y(T_{i+1}-T_i)\}\in(0,2P^{\varepsilon-1}]\cup[1-2P^{\varepsilon-1},1)$, by the argument above it follows that either $T_{i+1}-T_i\geq \frac{1}{2}\beta^{-1}M^{-1}$ or
\begin{align}\label{libo}
\{y(T_{i+1}-T_i)-(1-P^{\varepsilon-1})\}\geq \beta M.
\end{align}
But since $$\big(\lfloor P^{\varepsilon-1}\beta^{-1}M^{-1}\rfloor+1\big)\beta M>P^{\varepsilon-1},$$ there exists among the numbers $i=1,2,...,\lfloor P^{\varepsilon-1}\beta^{-1}M^{-1}\rfloor+1$ such a number $i$ that there holds $T_{i+1}-T_i\geq \frac{1}{2}\beta^{-1}M^{-1}$. Note also that \eqref{libo} implies the fact that among any $\lceil 2 P^{\varepsilon-1}\beta^{-1}M^{-1}\rceil+1$ consequitive values $i$ we can find such one that $T_{i+1}-T_i\geq \frac{1}{2}\beta^{-1}M^{-1}$.

Note that $\lfloor P^{\varepsilon-1}\beta^{-1}M^{-1}\rfloor M\leq P^{\varepsilon-1}P^{n-\varepsilon}=P^{n-1}$, hence, $$\mathfrak{T}:= \Big\{M,2M,...,\lfloor P^{\varepsilon-1}\beta^{-1}M^{-1}\rfloor M\Big\}\subset\{T_k\}.$$
Thus,
\begin{align*}
&P^{n-1}\geq T_K\geq \lfloor P^{\varepsilon-1}\beta^{-1}M^{-1}\rfloor M\\
&\qquad\qquad\qquad+\frac{K-\lfloor P^{\varepsilon-1}\beta^{-1}M^{-1}\rfloor}{\lceil 2 P^{\varepsilon-1}\beta^{-1}M^{-1}\rceil+1}\Big(\lceil 2 P^{\varepsilon -1}\beta^{-1}M^{-1}\rceil M+\frac{\beta^{-1}M^{-1}}{2}\Big)\\
&\qquad\qquad\qquad\qquad\qquad\qquad\geq \frac{K-\lfloor P^{\varepsilon-1}\beta^{-1}M^{-1}\rfloor}{3 P^{\varepsilon-1}\beta^{-1}M^{-1}}\frac{\beta^{-1}M^{-1}}{2},
\end{align*}
which yields
$$K-\lfloor P^{\varepsilon-1}\beta^{-1}M^{-1}\rfloor\leq 6P^{n-2+\varepsilon}.$$
Then, taking into account \eqref{lemma13} and Lemma \ref{lem1}, we establish
\begin{align*}
\#\Big\{(y_1,y_2,...,&y_{n-1})\in\{1,2,...,P\}^{n-1}:\|yy_1...y_{n-1}\|\leq P^{\varepsilon-1} \Big\}\\
= \#\Big\{(y_1,y_2,&...,y_{n-1})\in\{1,2,...,P\}^{n-1}:y_1...y_{n-1}\in\{T_k\}\Big\}\\
=\#\Big\{(y_1,&y_2,...,y_{n-1})\in\{1,2,...,P\}^{n-1}:y_1...y_{n-1}\in\{T_k\}\setminus\mathfrak{T}\Big\}\\
+\;\#\Big\{(&y_1,y_2,...,y_{n-1})\in\{1,2,...,P\}^{n-1}:y_1...y_{n-1}\in\mathfrak{T}\Big\}\\
\leq 6P^{n-2+\varepsilon}&C_{n-1}\Big(\frac{2\varepsilon}{n-1}\Big)(P^{n-1})^{\frac{2\varepsilon}{n-1}}\\
+\;\#\Big\{(y_1,y_2,&...,y_{n-1})\in\{1,2,...,P\}^{n-1}:\\
&\qquad\qquad\qquad y_1...y_{n-1}\in\big\{M,2M,...,\lfloor P^{\varepsilon-1}\beta^{-1}M^{-1}\rfloor M\big\}\Big\}\\
\leq 6P^{n-\frac{3}{2}}&C_{n-1}\Big(\frac{2\varepsilon}{n-1}\Big)\\
&\qquad +\;\#\Big\{(y_1,y_2,...,y_{n-1})\in\{1,2,...,P\}^{n-1}:y_1...y_{n-1}=M\Big\}\\
\cdot \;\#\Big\{(y_1,&y_2,...,y_{n-1})\in\{1,2,...,P\}^{n-1}:y_1...y_{n-1}\leq P^{\varepsilon-1}\beta^{-1}M^{-1}\Big\}\\
\leq &6P^{n-\frac{3}{2}}C_{n-1}\Big(\frac{2\varepsilon}{n-1}\Big)+C_{n-1}\Big(\frac{1}{2(n-1)}\Big)M^{\frac{1}{2(n-1)}}\frac{(n-1)P^{n-1}}{(\beta M P^{n-\varepsilon})^{\frac{1}{n-1}}}.
\end{align*}
So,
\begin{align*}
&\sum_{y_1,...,y_{n-1}=1}^P\min\bigg\{P,\frac{1}{2\|yy_1...y_{n-1}\|}\bigg\}\leq P \\
&\cdot \;\bigg(6P^{n-\frac{3}{2}}C_{n-1}\Big(\frac{2\varepsilon}{n-1}\Big)+(n-1)C_{n-1}\Big(\frac{1}{2(n-1)}\Big)P^{n-1-\frac{n-\varepsilon}{n-1}}\beta^{-\frac{1}{n-1}}M^{-\frac{1}{2(n-1)}}\bigg)\\
&\qquad\qquad\qquad\qquad\qquad+P^{n-1}\frac{1}{2 P^{\varepsilon-1}}\leq U\big(P^{n-\varepsilon}+P^{n-\frac{n-\varepsilon}{n-1}}\beta^{-\frac{1}{n-1}}M^{-\frac{1}{2(n-1)}}\big),
\end{align*}
where $U$ depends only on $n$ and $\varepsilon$.
\end{proof}

\begin{corollary}\label{cor2} Under the conditions of Lemma \ref{lm2}, for any real monic polynomial $f$ of degree $n$, there holds
\begin{align}\label{otbr}
\bigg|\sum_{k=1}^P e^{\frac{2\pi if(k)y}{n!}}\bigg|\leq D_1 \big(P^{1-\frac{\varepsilon}{2^{n-1}}}+P^{1-\frac{n-\varepsilon}{2^{n-1}(n-1)}}\beta^{-\frac{1}{2^{n-1}(n-1)}}M^{-\frac{1}{2^n(n-1)}}\big),
\end{align}
where $D_1$ depends only on $n$ and $\varepsilon$.
\end{corollary}

\begin{proof} From Lemma \ref{lm2} and \eqref{fin_S_m} it follows in the same way as \eqref{an} that
\begin{align*}
&\qquad\bigg|\sum_{k=1}^P e^{\frac{2\pi if(k)y}{n!}}\bigg|^{2^{n-1}}\leq 2^{2^{n-1}}P^{2^{n-1}-1}(n-1)+\\
&+2^{2^{n-1}+1}P^{2^{n-1}-n}\big(UP^{n-\varepsilon}+UP^{n-\frac{n-\varepsilon}{n-1}}\beta^{-\frac{1}{n-1}}M^{-\frac{1}{2(n-1)}}\big)\\
&\qquad\qquad\qquad\qquad\qquad\leq D_1'P^{2^{n-1}-\varepsilon}+D_1'P^{2^{n-1}-\frac{n-\varepsilon}{n-1}}\beta^{-\frac{1}{n-1}}M^{-\frac{1}{2(n-1)}},
\end{align*}
therefore, in view of the inequality $(a+b)^{\frac{1}{n-1}}\leq a^{\frac{1}{n-1}}+b^{\frac{1}{n-1}}$, which is valid for any positive $a$ and $b$, we get the needed assertion with $D_1=(D_1')^{\frac{1}{2^{n-1}}}$. 
 
\end{proof}

\begin{lemma}\label{lem3} Let $0<y\in\mathbb{R}\setminus\mathbb{Q},\;\varepsilon\in(0,\frac{1}{6}),\;8\leq P\in\mathbb{N},\; 3\leq n\in\mathbb{N}$. If there exists such a pair of coprime natural numbers $C$ and $M\leq P^{\varepsilon}$ that
$$\Big|y-\frac{C}{M}\Big|=:|\beta|\leq P^{\varepsilon-n},$$
then there holds
$$\sum_{y_1,...,y_{n-1}=1}^P\min\bigg\{P,\frac{1}{2\|yy_1...y_{n-1}\|}\bigg\}\leq \frac{BP^n}{M^{\frac{1}{2(n-1)}}},$$
where $B$ depends only on $n$. Besides, in the case $P^n|\beta|\geq 2$, for any $\delta>0$ there holds
$$\sum_{y_1,...,y_{n-1}=1}^P\min\bigg\{P,\frac{1}{2\|yy_1...y_{n-1}\|}\bigg\}\leq \frac{P^{n-\varepsilon}}{2}+\frac{A_{\delta}}{M^{\frac{1}{2(n-1)}}|\beta|^{\frac{1}{n-1}-\delta}}P^{n-\frac{n}{n-1}+\delta n},$$
where $A_{\delta}$ depends only on $\delta$ and $n$.
\end{lemma}

\begin{proof}
Without loss of generality assume that $\beta>0$. Let us show that the minimal $T\in\mathbb{N}$ such that $\|yT\|\leq P^{\varepsilon-1}$ is $M$. Indeed, otherwise there exists $T<M$ such that $yT=Z+\gamma$, where $Z\in\mathbb{N}\cup\{0\}, \gamma\in (0,P^{\varepsilon-1}]\cup[1-P^{\varepsilon-1},1)$, and in this case
$$\frac{Z}{T}+\frac{\gamma}{T}=y=\frac{C}{M}+\beta.$$
On the other hand,
\begin{align}\label{todo3}
\frac{1}{P^{2\varepsilon}}\leq \frac{1}{M^2}\leq \frac{1}{TM}\leq\Big|\frac{C}{M}-\frac{Z}{T}\Big|=\Big|\beta-\frac{\gamma}{T}\Big|<\frac{2}{P^{1-\varepsilon}},
\end{align}
so $P^{1-3\varepsilon}<2$ which is false, since $P^{1-3\varepsilon}\geq \sqrt{P}> 2$. Let $T_1=M<T_2<T_3<...<T_K$ be all the natural numbers less than $P^{n-1}$ such that $\|yT_k\|\leq P^{\varepsilon-1}$. Let us show that $K=\lfloor P^{\varepsilon-1}\beta^{-1}M^{-1}\rfloor$ and that $T_k=kM$ for each $k\leq K$. We will prove the second part of the assertion by induction. The basis is the case $k=1$. Suppose that the assertion is proved for $k=l<K$ and let us prove it for $k=l+1$. Assume the contrary, $lM<T_{l+1}<(l+1)M$. Due to irrationality of $y$ we have $\{yT_{l+1}\}\neq\{ylM\}=l\beta$. The only two remaining cases are the following ones:

1) $l\beta<\{yT_{l+1}\}\leq P^{\varepsilon-1}$. But then $0<T_{l+1}-lM<M$ and $\{y(T_{l+1}-lM)\}\in(0,P^{\varepsilon-1}]$, which is not possible, since $T_1=M$.

2) $\{yT_{l+1}\}\in[1-P^{\varepsilon-1},1)$ or $\{yT_{l+1}\}<l\beta\leq P^{n-1}P^{\varepsilon-n}=P^{\varepsilon-1}$. Then $\{y(T_{l+1}-lM)\}\in(1-2P^{\varepsilon-1},1)$, hence,
\begin{align}\label{todo4}
\{y(T_{l+1}-lM)M\}>1-2MP^{\varepsilon-1}\geq 1-2P^{2\varepsilon-1}\geq 1-2P^{-2/3}>\frac{1}{2}.
\end{align}
On the other hand
\begin{align}\label{todo5}
\{yM(T_{l+1}-lM)\}\leq (T_{l+1}-lM)M\beta\leq M^2\beta\leq P^{3\varepsilon-n}<P^{-2}\leq \frac{1}{64}.
\end{align}
We have a contradiction. It is only left to show that $K=\lfloor P^{\varepsilon-1}\beta^{-1}M^{-1}\rfloor$. Note that for $l$ satisfying $\lfloor P^{\varepsilon-1}\beta^{-1}M^{-1}\rfloor+1\leq l\leq P^{n-1}$ there holds
$$\{ylM\}\leq \{yM\}l\leq P^{n-1}\beta M\leq P^{n-1-n+\varepsilon+\varepsilon}=P^{2\varepsilon-1}$$
and
$$\{ylM\}> P^{\varepsilon-1}\beta^{-1}M^{-1}\beta M=P^{\varepsilon-1},$$
i.e., $lM\neq T_k$ for any $k$. Suppose there exists such a natural number $T$ that $lM<T<(l+1)M$ for some $\lfloor P^{\varepsilon-1}\beta^{-1}M^{-1}\rfloor\leq l\leq P^{n-1}$ and $\|yT\|\leq P^{\varepsilon-1}$. But then $\{y(T-lM)\}\geq 1- 2P^{2\varepsilon-1}$, therefore, once more
\begin{align}\label{todo6}
\{y(T-lM)M\}>1-2MP^{2\varepsilon-1}\geq 1-2P^{3\varepsilon-1}\geq 1-2P^{-1/2}\geq 1-\frac{1}{\sqrt{2}}\geq\frac{1}{5}.
\end{align}
In contrast,
$$\{yM(T-lM)\}\leq (T-lM)M\beta\leq M^2\beta\leq P^{3\varepsilon-3}<P^{-2}\leq \frac{1}{64},$$
which leads to a contradiction. Thus, 
\begin{align}\label{K}
K=\lfloor P^{\varepsilon-1}\beta^{-1}M^{-1}\rfloor \qquad\text{and}\qquad T_k=kM \;\; \text{for all}\;\; k\leq K. 
\end{align}

Therefore,
\begin{align}\label{1sum}
\sum_{\underset{y_1...y_{n-1}\notin \{T_k\}}{y_1,...,y_{n-1}=1}}^P\min\bigg\{P,\frac{1}{2\|yy_1...y_{n-1}\|}\bigg\}\leq P^{n-1}\frac{1}{2 P^{\varepsilon-1}}\leq \frac{P^{n-\varepsilon}}{2}.
\end{align}

Note that using \eqref{lemma13} we have, for an arbitrary natural $l$,
\begin{align}\label{multiple}
\#\Big\{(y_1,y_2,...,&y_{n-1})\in \{1,2,...,P\}^{n-1}:y_1y_2...y_{n-1}\in\{M,2M,...,lM\}\Big\}\nonumber\\
\quad\quad\leq \#\Big\{(y_1,&y_2,...,y_{n-1})\in \{1,2,...,P\}^{n-1}:y_1y_2...y_{n-1}\leq l\Big\}\nonumber\\
\cdot \;\#\Big\{(y_1&,y_2,...,y_{n-1})\in \{1,2,...,P\}^{n-1}:y_1y_2...y_{n-1}=M\Big\}\nonumber\\
\leq &C_{n-1}\Big(\frac{1}{2(n-1)}\Big)M^{\frac{1}{2(n-1)}}\nonumber\\
&\cdot\;\#\Big\{(y_1,y_2,...,y_{n-1})\in \{1,2,...,P\}^{n-1}:y_1y_2...y_{n-1}\leq l\Big\}.
\end{align}

Now, according to \eqref{multiple} and Lemma \ref{lem1}, we can estimate
\begin{align}\label{3sum}
&\sum_{\underset{y_1...y_{n-1}\in\{T_k\}}{y_1,...,y_{n-1}=1}}^P\min\bigg\{P,\frac{1}{2\|yy_1...y_{n-1}\|}\bigg\}\leq \sum_{\underset{y_1...y_{n-1}\in\{T_k\}}{y_1,...,y_{n-1}=1}}^P P\leq P C_{n-1}\Big(\frac{1}{2(n-1)}\Big)\nonumber\\
&\cdot M^{\frac{1}{2(n-1)}}\#\bigg\{(y_1,y_2,...,y_{n-1})\in \{1,2,...,P\}^{n-1}:y_1y_2...y_{n-1}\leq \frac{P^{n-1}}{M}\bigg\}\nonumber\\
&\qquad\qquad\qquad\qquad\qquad\qquad\qquad\qquad\leq C_{n-1}\Big(\frac{1}{2(n-1)}\Big)\frac{(n-1)P^n}{M^{\frac{1}{2(n-1)}}}.
\end{align}
Combining estimates \eqref{1sum} and \eqref{3sum}, we derive
\begin{align*}
\sum_{y_1,...,y_{n-1}=1}^P\min\bigg\{P,\frac{1}{2\|yy_1...y_{n-1}\|}\bigg\}&\leq \frac{P^{n-\varepsilon}}{2}+C_{n-1}\Big(\frac{1}{2(n-1)}\Big)\frac{(n-1)P^n}{M^{\frac{1}{2(n-1)}}}\\
&<\frac{BP^n}{M^{\frac{1}{2(n-1)}}}.
\end{align*}

Consider now the case $P^n\beta\geq 2$. Let
$$F_t:=\bigg\{(y_1,...,y_{n-1})\in \{1,2,...,P\}^{n-1}:y_1...y_{n-1}\in\Big\{M,2M,...,\left\lfloor t(P\beta M)^{-1} \right\rfloor M\Big\}\bigg\}.$$
Note that for $t>P^n\beta$ there holds
$$\left\lfloor \frac{t}{P\beta M} \right\rfloor M > \frac{P^n\beta}{P\beta M}M=P^{n-1}\geq y_1y_2...y_{n-1},$$
and hence, for such $t$ we have $F_t=F_{\lfloor P^n\beta \rfloor}$. For any natural $t\leq P^n\beta$, due to Lemma \ref{lem1}, we get
\begin{align}\label{helplem1}
\#\bigg\{(y_1,...,y_{n-1})\in \{1,2,...,P\}^{n-1}:y_1...y_{n-1}\leq \frac{t}{P\beta M}\bigg\}\leq \frac{(n-1)P^{n-1}}{\big(\frac{P^n\beta M}{t}\big)^{\frac{1}{n-1}}}.
\end{align}
Summing up \eqref{multiple} and \eqref{helplem1}, we have
\begin{align}\label{Ft}
&|F_t|\leq C_{n-1}\Big(\frac{1}{2(n-1)}\Big)M^{\frac{1}{2(n-1)}}\frac{(n-1)P^{n-1}}{\big(\frac{P^n\beta M}{t}\big)^{\frac{1}{n-1}}}\nonumber\\
&\qquad\qquad\qquad\qquad\qquad\qquad=C_{n-1}\Big(\frac{1}{2(n-1)}\Big)\frac{(n-1)P^{n-1}}{\big(\frac{P^n\beta M^{1/2}}{t}\big)^{\frac{1}{n-1}}}.
\end{align}

If $y_1y_2...y_{n-1}\in\{T_k\}_{k=1}^K\setminus F_t$, then
\begin{align}\label{logest}
\frac{1}{2\|yy_1...y_{n-1}\|}\leq \frac{1}{2\beta t P^{-1}\beta^{-1}}= \frac{P}{2t}.
\end{align}
Using \eqref{logest}, we obtain
\begin{align*}
&S(P,y):=\sum_{\underset{y_1...y_{n-1}\in\{T_k\}}{y_1,...,y_{n-1}=1}}^P\min\bigg\{P,\frac{1}{2\|yy_1...y_{n-1}\|}\bigg\}\leq |F_2|\max p+\big(|F_3|-|F_2|\big)\frac{P}{2\cdot 2}\\
&\qquad\qquad\qquad+\big(|F_4|-|F_3|\big)\frac{P}{2\cdot 3}+...+\big(|F_{\lfloor P^n\beta \rfloor+1}|-|F_{\lfloor P^n\beta \rfloor}|\big)\frac{P}{2\cdot {\lfloor P^n\beta \rfloor}}.
\end{align*}
Noting that in the expression above any term $|F_i|$ appears with a nonnegative coefficient we derive with the help of \eqref{Ft}
\begin{align}\label{diffs}
&S(P,y)\leq C_{n-1}\Big(\frac{1}{2(n-1)}\Big)P(n-1)P^{n-1-\frac{n}{n-1}}\beta^{-\frac{1}{n-1}}M^{-\frac{1}{2(n-1)}}\\
&\cdot\bigg(2^{\frac{1}{n-1}}+\Big(3^{\frac{1}{n-1}}-2^{\frac{1}{n-1}}\Big)\frac{1}{2}+...+\Big(\big(\lfloor P^n\beta \rfloor +1\big)^{\frac{1}{n-1}}-\lfloor P^n\beta \rfloor^{\frac{1}{n-1}}\Big)\frac{1}{\lfloor P^n\beta\rfloor}\bigg).\nonumber
\end{align}
For any $a\geq 2$ by Lagrange's theorem we have for some $\theta\in(0,1)$
$$(a+1)^{\frac{1}{n-1}}-a^{\frac{1}{n-1}}=\frac{1}{n-1}(a+\theta)^{-\frac{n-2}{n-1}}\leq \frac{2}{n-1}a^{-\frac{n-2}{n-1}}\leq 1\leq 2^{\frac{1}{n-1}},$$
and hence, it follows from \eqref{diffs} that
\begin{align}\label{estdiffs}
S(P,y)\leq C_{n-1}\Big(\frac{1}{2(n-1)}\Big)P(n-1)P^{n-1-\frac{n}{n-1}}\beta^{-\frac{1}{n-1}}M^{-\frac{1}{2(n-1)}}2^{\frac{1}{n-1}}\ln(P^n\beta).
\end{align}
Since for any $\delta>0$ there exists such a number $B_{\delta}>0$ that $\ln(s+1)\leq B_{\delta} s^{\delta}$ for $s>0$, \eqref{estdiffs} implies
\begin{align}\label{2sum}
S(P,y)&\leq B_{\delta}C_{n-1}\Big(\frac{1}{2(n-1)}\Big)(n-1)P^{n-\frac{n}{n-1}+\delta n}\beta^{\delta-\frac{1}{n-1}}M^{-\frac{1}{2(n-1)}}2^{\frac{1}{n-1}}\nonumber\\
&=A_{\delta}P^{n-\frac{n}{n-1}+\delta n}\beta^{\delta-\frac{1}{n-1}}M^{-\frac{1}{2(n-1)}},
\end{align}
where $A_{\delta}$ is a constant depending only on $n$. Combining estimates \eqref{1sum} and \eqref{2sum}, we get the desired inequality for the case $P^n\beta\geq 2$. 
\end{proof}

\begin{corollary}\label{cor3} Under the conditions of Lemma \ref{lem3}, for any real monic polynomial $f$ of degree $n$, there holds
\begin{align}\label{after_sqr}
\bigg|\sum_{k=1}^P e^{\frac{2\pi if(k)y}{n!}}\bigg|\leq W\frac{P}{M^{\frac{1}{2^n(n-1)}}},
\end{align}
where $W$ depends only on $n$. Besides, if $P^n|\beta|\geq 2$, then we have
\begin{align}\label{sqr}
\bigg|\sum_{k=1}^P e^{\frac{2\pi if(k)y}{n!}}\bigg|\leq J\bigg(P^{1-\frac{\varepsilon}{2^{n-1}}}+\frac{P^{1-\frac{n}{2^n(n-1)}}}{(M|\beta|)^{\frac{1}{2^n(n-1)}}}\bigg),
\end{align}
where $J$ depends only on $n$.
\end{corollary}

\begin{proof}

By Lemma \ref{lem3} and \eqref{fin_S_m}, we derive
\begin{align*}
\bigg|\sum_{k=1}^P e^{\frac{2\pi if(k)y}{n!}}\bigg|^{2^{n-1}}\leq 2^{2^{n-1}}P^{2^{n-1}-n}\bigg((n-1)P^{n-1}+2\frac{BP^n}{M^{\frac{1}{2(n-1)}}}\bigg),
\end{align*}
which implies \eqref{after_sqr}.

If $P^n|\beta|\geq 2$, then, according to \eqref{fin_S_m} and Lemma \ref{lem3},
\begin{align}\label{2}
\bigg|\sum_{k=1}^P e^{\frac{2\pi if(k)y}{n!}}\bigg|^{2^{n-1}}&\leq 2^{2^{n-1}}P^{2^{n-1}-1}(n-1)\\
&+2^{2^{n-1}+1}P^{2^{n-1}-n}\bigg(\frac{P^{n-\varepsilon}}{2}+\frac{A_{\frac{1}{2(n-1)}}}{M^{\frac{1}{2(n-1)}}|\beta|^{\frac{1}{2(n-1)}}}P^{n-\frac{n}{2(n-1)}}\bigg).\nonumber
\end{align}
Using the inequality $(a+b)^{\frac{1}{n-1}}\leq a^{\frac{1}{n-1}}+b^{\frac{1}{n-1}}$, $a,b>0$, and \eqref{2} we get \eqref{sqr}.
\end{proof}

\begin{remark}\label{imp} Lemmas \ref{lm2} and \ref{lem3}, as well as Corollaries \ref{cor2} and \ref{cor3}, remain true for $\varepsilon\in(0,\frac{1}{3})$ for large enough $P$.
\end{remark}

\begin{proof}
There are only the following estimates which need apparent changes: \eqref{todo1}, \eqref{todo2}, \eqref{todo3}, \eqref{todo4}, \eqref{todo5} and \eqref{todo6}.
\end{proof}

\section{The case of a natural power}

Consider the case $\alpha=:n\in\mathbb{N}$. Here we have two fundamentally different situations: when $n$ is even and when $n$ is odd. We start with the even case.

\begin{proof}[Proof of Theorem \ref{thm} (a)] 
Note that $k^2\equiv 0\;(\!\!\!\!\mod 4)$ for any even $k$ and $k^2\equiv 1\;(\!\!\!\!\mod 4)$ for any odd $k$. Therefore, for any $l<L$,
\begin{align}\label{00}
\sum_{k=2l+1}^{2L}c_k\sin k^{n} \frac{\pi}{2}=\sum_{k=1}^{L-l}c_{2l+2k-1}\geq \frac{1}{2}\sum_{k=2l+1}^{2L}c_k,
\end{align}
which completes the proof. 

A similar argument is valid for the point $\frac{2\pi}{3}$, due to the fact that $k^2\equiv 1\;(\!\!\!\!\mod 3)$ for any $k$ not divisible by $3$.
\end{proof}  

\vspace{4pt}

Let us turn now to the case of an odd power $n$ and series \eqref{series2}. This case differs from the previous one by the fact that for any odd $l$ and any point of the form $x=2\pi\frac{a}{b}\in 2\pi\mathbb{Q}$ there holds $\sum_{k=1}^b\sin k^l x=0$. Therefore, there is no $``$accomulation$"$ which we saw in the case of an even $l$. We will see that due to this property we manage to get suitable estimates of the corresponding image part of Weyl sums for points close enough to $\pi$-rationals. For other points we obtain effective estimates provided by Section \ref{Weyl}, and these estimates are still valid if we replace $f$ by any polynomial of the same degree.

\begin{proof}[Proof of Theorem \ref{thm} (b)]
Let $c_kk\to 0$. We fix some $\varepsilon\in(0,\frac{1}{6})$ and $x\in\mathbb{R}\setminus \pi\mathbb{Q}$. Since we are going to prove the assertion for all $x\in\mathbb{R}$, we can assume that the polynomial $f$ is monic. Without loss of generality we consider $x>0$. Denote
$$\mathfrak{M}=\mathfrak{M}_x:=\bigg\{M\in\mathbb{N}:\exists C_M\in\mathbb{N}\;\;\text{such that}\;\;(C_M,M)=1\;\;\text{and}\;\;\Big|\frac{n!x}{2\pi}-\frac{C_M}{M}\Big|\leq M^{\frac{\varepsilon-1}{\varepsilon}}\bigg\}.$$
Notice that $\mathfrak{M}$ is a finite or infinite increasing sequence of natural numbers $\{M_i\}_{i\geq 1}$, and there holds
\begin{align}\label{M_i}
2M_{i+1}\geq M_i^{4}
\end{align} 
for any $i\geq 1$. Indeed,
\begin{align*}
\frac{1}{M_iM_{i+1}}\leq\Big|\frac{C_{M_i}}{M_i}-\frac{C_{M_{i+1}}}{M_{i+1}}\Big|\leq \Big|\frac{n!x}{2\pi}-\frac{C_{M_i}}{M_i}\Big|&+\Big|\frac{n!x}{2\pi}-\frac{C_{M_{i+1}}}{M_{i+1}}\Big|\\
&\leq M_i^{\frac{\varepsilon-1}{\varepsilon}}+M_{i+1}^{\frac{\varepsilon-1}{\varepsilon}}\leq 2M_i^{\frac{\varepsilon-1}{\varepsilon}},
\end{align*}
therefore,
$$2M_{i+1}\geq M_i^{\frac{1-\varepsilon}{\varepsilon}-1}\geq M_i^{\frac{1-1/6}{1/6}-1}=M_i^4.$$
We call a natural number $m$ {\it inconvinient} if there exists a pair of coprime natural numbers $C$ and $M\leq m^{\varepsilon}$ such that 
$$\Big|\frac{n!x}{2\pi}-\frac{C}{M}\Big|\leq m^{\varepsilon-n}.$$
Otherwise, if there exists a pair of coprime natural numbers $C$ and $M\leq m^{\varepsilon}$ such that
$$\Big|\frac{n!x}{2\pi}-\frac{C}{M}\Big|\leq m^{\varepsilon-1},$$
we say that $m$ is {\it almost convenient}. In other cases we call $m$ {\it convenient}. Let
\begin{align*}
S_m(x):=\Im\sum_{k=1}^m e^{i f(k)x}=\sum_{k=1}^m\sin f(k)x.
\end{align*}

For any natural $l<L$, using the Abel transformation we have
\begin{align}\label{simple_sum}
\bigg|\sum_{m=l}^L c_m\sin f(m)x\bigg|\leq c_l l&+c_{L+1}L+\bigg|\sum_{m=l}^{L}(c_m-c_{m+1})S_m(x)\bigg|\qquad\qquad\qquad\nonumber\\
&\leq 2\sup_{k\geq l}c_k k+\sum_{m=l}^{L}(c_m-c_{m+1})|S_m(x)|.
\end{align}

From now on we assume $l\geq 9$ in order to apply estimates from Section \ref{Weyl}. With the help of Corollaries \ref{cor1} and \ref{cor2} we will estimate $S_m(x)$ for convenient and almost convenient $m$, taking into account that for these cases $\sin\frac{n!y_1...y_{n-1}x}{2\pi}$, roughly speaking, rarely assumes values close to zero. For such $m$ we do not use oddness of $f$ and the fact that the coefficients of $f$ are rational.

For convenient $m$, using Corollary \ref{cor1} we get
\begin{align}\label{udobn}
&\sum_{\underset{\text{convenient}}{m\geq l}}(c_m-c_{m+1})|S_m(x)|\leq D\sum_{m\geq l}(c_m-c_{m+1})m^{1-\frac{\varepsilon}{2^n}}\leq Dc_l l^{1-\frac{\varepsilon}{2^n}}\nonumber\\
&\quad+2D\sum_{m\geq l}c_m m^{-\frac{\varepsilon}{2^n}}\leq D\Big(1+2\sum_{m\geq l} m^{-1-\frac{\varepsilon}{2^n}}\Big)\sup_{k\geq l}c_k k\nonumber\\
&\qquad\leq D\Big(1+\frac{2^{n+1}}{\varepsilon}\Big)\sup_{k\geq l}c_k k\leq D\frac{2^{n+2}}{\varepsilon}\sup_{k\geq l}c_k k=:D''\sup_{k\geq l}c_k k.
\end{align}

Note that there is a natural number $M=M(m)\leq m^{\varepsilon},\;M\in\mathfrak{M}$, kept for any almost convenient $m$, and all the numbers $m$ for which a certain $M$ is kept must satisfy the condition 
\begin{align}\label{proneoch}
\Big|\frac{n!x}{2\pi}-\frac{C}{M}\Big|^{-\frac{1}{n-\varepsilon}}=:\beta^{-\frac{1}{n-\varepsilon}}\leq m\leq\beta^{-\frac{1}{1-\varepsilon}}.
\end{align}
For almost convenient $m$, according to \eqref{proneoch} and Corollary \ref{cor2}, we have
\begin{align*}
&\sum_{\underset{\text{almost convenient}}{m\geq l}}(c_m-c_{m+1})|S_m(x)|\leq \sum_{M\in \mathfrak{M}}\sum_{\underset{\text{almost convenient}}{m\geq l,\;M=M(m)}}(c_m-c_{m+1})|S_m(x)|\\
&\leq D_1\sum_{m\geq l}(c_m-c_{m+1})m^{1-\frac{\varepsilon}{2^{n-1}}}\\
&\qquad+D_1\beta^{-\frac{1}{2^{n-1}(n-1)}}\sum_{M\in \mathfrak{M}}M^{-\frac{1}{2^n(n-1)}} \sum_{\underset{m=\lceil\beta^{-\frac{1}{n-\varepsilon}}\rceil}{m\geq l}}^{\lfloor\beta^{-\frac{1}{1-\varepsilon}}\rfloor}(c_m-c_{m+1})m^{1-\frac{n-\varepsilon}{2^{n-1}(n-1)}}.
\end{align*}
We estimate the first sum as in \eqref{udobn}, so
\begin{align}\label{taksebe}
&\sum_{\underset{\text{almost convenient}}{m\geq l}}(c_m-c_{m+1})|S_m(x)|\leq D''\frac{D_1}{D}\sup_{k\geq l}c_k k\nonumber\\
&+2D_1\beta^{-\frac{1}{2^{n-1}(n-1)}}\sum_{M\in \mathfrak{M}}M^{-\frac{1}{2^n(n-1)}} \sum_{\underset{m=\lceil\beta^{-\frac{1}{n-\varepsilon}}\rceil}{m\geq l}}^{\lfloor\beta^{-\frac{1}{1-\varepsilon}}\rfloor}c_m m^{-\frac{n-\varepsilon}{2^{n-1}(n-1)}}\leq D''\frac{D_1}{D}\sup_{k\geq l}c_k k\nonumber\\
&+2D_1\beta^{-\frac{1}{2^{n-1}(n-1)}}\sum_{M\in \mathfrak{M}}M^{-\frac{1}{2^n(n-1)}} \frac{2^{n-1}(n-1)}{n-\varepsilon}\big(\lceil\beta^{-\frac{1}{n-\varepsilon}}\rceil\big)^{-\frac{n-\varepsilon}{2^{n-1}(n-1)}}\sup_{k\geq l}c_k k\nonumber\\
&\qquad\qquad\qquad\qquad\qquad\qquad\qquad\qquad\qquad\qquad\qquad\qquad\leq D_1''\sup_{k\geq l}c_k k,
\end{align}
where $D_1''$ depends only on $n$ and $\varepsilon$ due to \eqref{M_i}.

Turn now to inconvenient numbers $m$. In this case $x$ is approximated by a $\pi$-rational fraction too well, therefore, the difference between $S_m$ at the point $x$ and $S_m$ at the close $\pi$-rational point is sufficiently small. For some $m$ we will use it, and this will be the unique case when we care about the fact that we have to estimate only the imaginary part of the Weyl sum, the fact that $f$ is an odd function and that the coefficients of $f$ are rational. For other $m$ we will proceed according one of the two inequalities from Corollary \ref{cor3}, and these estimates will remain true for the whole Weyl sums and for any monic polynomial of the same degree.

There is also a number $M\leq m^{\varepsilon},\;M\in\mathfrak{M}$, kept for any inconvenient number $m$. A fixed $M\in\mathfrak{M}$ is kept for all natural numbers $m$ such that $\beta^{-\frac{1}{n-\varepsilon}}=\beta_M^{-\frac{1}{n-\varepsilon}}=\left|\frac{n!x}{2\pi}-\frac{C}{M}\right|^{-\frac{1}{n-\varepsilon}}\geq m\geq M^{\frac{1}{\varepsilon}}$, where $C=C_M$, and only for these numbers. Denote $m_1:=\lceil M^{\frac{1}{\varepsilon}}\rceil$ and $m_2:=\big\lfloor \beta^{-\frac{1}{n-\varepsilon}}\big\rfloor$ (for a fixed $M$). Let $K$ be such a natural number that 
\begin{align}\label{beta}
\frac{1}{K^n\ln (2M)}\leq\beta\leq\frac{1}{(K-1)^n \ln (2M)}.
\end{align}

We divide the interval $m_1\leq m\leq m_2$ into three intervals: $m_1\leq m\leq K-1,\;K\leq m\leq \lfloor 2K\ln (2M) \rfloor$ and $\lfloor 2K\ln (2M) \rfloor +1\leq m\leq m_2$. For $m$ belonging to the second or to the third one, we will estimate $S_m(x)$ with the help of \eqref{after_sqr} and \eqref{sqr}, respectively. And it is only on the interval $m_1\leq m\leq K-1$ where we will need for estimating $S_m(x)$ the properties of the polynomial $f$ and the fact that we deal with sines and not with cosines. 

Let $Q\in\mathbb{N}$ be the minimal number such that $Qf\in\mathbb{Z}[x]$. Note that for any odd $l$, if numbers $k_1$ and $k_2$ are such that $k_1\equiv -k_2\;(\!\!\!\!\mod QMn!)$, then $\sin \big(k_1^l \frac{2\pi C}{QMn!}\big)+\sin\big(k_2^l \frac{2\pi C}{QMn!}\big)=0$. Hence, for any $g\in \mathbb{Z}$, there holds $$\sum_{k=g+1}^{g+QMn!}\sin  \Big(f(k)\frac{2\pi C}{Mn!}\Big)=0,$$ which implies
\begin{align}\label{eps_4}
|S_m(x)|&\leq \Big|S_m\Big(\frac{2\pi C}{Mn!}\Big)\Big|+\Big|S_m(x)-S_m\Big(\frac{2\pi C}{Mn!}\Big)\Big|\leq \bigg|\sum_{k=1}^m\sin \Big(f(k)\frac{2\pi C}{Mn!}\Big) \bigg|\nonumber\\
&+\frac{2\pi}{n!}\sum_{k=1}^m |f(k)|\beta\leq \Big\{\frac{m}{QMn!}\Big\}QMn!+2C_f\beta\sum_{k=1}^m k^n \nonumber\\
&\qquad\;\;\quad\quad\quad\quad\quad\leq QMn!+C_fm^{n+1}\beta\leq Qm^{\varepsilon}n!+C_fm^{n+1}\beta,
\end{align}
where $C_f$ stands for the sum of the absolute values of the coefficients of $f$. From \eqref{beta} and \eqref{eps_4} we have
\begin{align}\label{sum_1_1}
\sum_{m=m_1}^{K-1}(c_m-c_{m+1})|S_m(x)|&\leq Qn!\sum_{m=m_1}^{K-1}(c_m-c_{m+1})m^{\varepsilon}+C_f\beta\sum_{m=m_1}^{K-1}(c_m-c_{m+1})m^{n+1}\nonumber\\
\leq Qn!c_{m_1} m_1^{\varepsilon}&+Qn!\sum_{m=m_1+1}^{K-1}c_m \varepsilon (m-\theta_m)^{\varepsilon-1}\nonumber\\
&+\frac{C_f}{(K-1)^n\ln (2M)}\sum_{m=m_1}^{K-1}(c_m-c_{m+1})m^{n+1},
\end{align}
where $\theta_m\in(0,1)$ by Lagrange's theorem. Taking into account that for any $z\geq 2$ there is the inequality $(z-1)^{\varepsilon-1}\leq 2z^{\varepsilon-1}$, it follows from \eqref{sum_1_1} that
\begin{align}\label{sum_1_2}
&\sum_{m=m_1}^{K-1}(c_m-c_{m+1})|S_m(x)|\leq Qn!m_1^{\varepsilon-1}\sup_{k\geq l}c_k k+2\varepsilon Qn!\sum_{m=m_1+1}^{K-1}c_m m^{\varepsilon-1}\nonumber\\
&+C_f\frac{c_{m_1} m_1}{\ln (2M)}+C_f\frac{n+1}{(K-1)^n\ln (2M)}\sum_{m=m_1+1}^{K-1}c_m m^n\leq \sup_{k\geq l}c_k k\nonumber\\
&\cdot\bigg(Qn!m_1^{\varepsilon-1}+2\varepsilon Qn!\sum_{m=m_1+1}^{K-1}m^{\varepsilon-2}+\frac{C_f}{\ln (2M)}+\frac{(n+1)C_f}{(K-1)^n\ln (2M)}\sum_{m=m_1+1}^{K-1}m^{n-1}\bigg)\nonumber\\
&\leq \left(Qn!\Big(m_1^{\varepsilon-1}+\frac{2\varepsilon}{1-\varepsilon}m_1^{\varepsilon-1}\Big)+\frac{C_f}{\ln (2M)}+\frac{(n+1)C_f}{(K-1)^n\ln (2M)}\frac{K^n}{n}\right)\sup_{k\geq l}c_k k,
\end{align}
from which, due to the fact that $m_1^{\varepsilon}\geq M$ and that $\frac{(n+1)K^n}{n(K-1)^n}\leq 2^{n+1}$, we finally derive
\begin{align}\label{sum_1}
\sum_{m=m_1}^{K-1}(c_m-c_{m+1})|S_m(x)|\leq A\Big(M^{\frac{\varepsilon-1}{\varepsilon}}+\frac{1}{\ln (2M)}\Big)\sup_{k\geq l}c_k k,
\end{align}
where $A>0$ depends only on $n,\varepsilon$ and $f$.
\\

For $2K\ln (2M)\leq m\leq m_2$, we have

$$m^n\beta\geq \big(2K\ln (2M)\big)^n\big(K^n\ln (2M)\big)^{-1}=2^n\ln^{n-1}(2M)\geq 2(2\ln 2)^{n-1}\geq 2.$$
Hence, using \eqref{sqr},
\begin{align}\label{sum_3_1}
&\sum_{m=\lfloor 2K\ln (2M) \rfloor +1}^{m_2}(c_m-c_{m+1})|S_m(x)| \nonumber\\
&\qquad\leq J\sum_{m=\lfloor 2K\ln (2M) \rfloor +1}^{m_2}(c_m-c_{m+1})\bigg(m^{1-\frac{\varepsilon}{2^{n-1}}}+\frac{m^{1-\frac{n}{2^n(n-1)}}}{(M\beta)^{\frac{1}{2^n(n-1)}}}\bigg).
\end{align}
First, we estimate
\begin{align}\label{sum_3_2}
&\sum_{m=\lfloor 2K\ln (2M) \rfloor +1}^{m_2}(c_m-c_{m+1})m^{1-\frac{\varepsilon}{2^{n-1}}} \nonumber\\
&\leq\frac{c_{\lfloor 2K\ln (2M) \rfloor +1}\big(\lfloor 2K\ln (2M) \rfloor +1\big)}{\big(\lfloor 2K\ln (2M) \rfloor +1\big)^{\frac{\varepsilon}{2^{n-1}}}}+2\sum_{m=\lfloor 2K\ln (2M) \rfloor +2}^{m_2}c_m m^{-\frac{\varepsilon}{2^{n-1}}}\nonumber\\
&\qquad\leq\bigg(m_1^{-\frac{\varepsilon}{2^{n-1}}}+2\sum_{m=\lfloor 2K\ln (2M) \rfloor +2}^{m_2}m^{-1-\frac{\varepsilon}{2^{n-1}}}\bigg)\sup_{k\geq l}c_k k\nonumber\\
&\qquad\qquad\leq \Big(M^{-\frac{1}{2^{n-1}}}+\frac{2^n}{\varepsilon}m_1^{-\frac{\varepsilon}{2^{n-1}}}\Big)\sup_{k\geq l}c_k k\leq\frac{2^{n+1}}{\varepsilon}M^{-\frac{1}{2^{n-1}}}\sup_{k\geq l}c_k k.
\end{align}
Further, in view of \eqref{beta},
\begin{align}\label{sum_3_3}
&\sum_{m=\lfloor 2K\ln (2M) \rfloor +1}^{m_2}(c_m-c_{m+1})\frac{m^{1-\frac{n}{2^n(n-1)}}}{(M\beta)^{\frac{1}{2^n(n-1)}}}\nonumber\\
&\leq \sum_{m=\lfloor 2K\ln (2M) \rfloor +1}^{m_2}(c_m-c_{m+1})\frac{m^{1-\frac{n}{2^n(n-1)}}K^{\frac{n}{2^n(n-1)}}\big(\ln (2M)\big)^{\frac{1}{2^n(n-1)}}}{M^{\frac{1}{2^n(n-1)}}}\nonumber\\
&\qquad\leq \Big(\frac{\ln (2M)}{M}\Big)^{\frac{1}{2^n(n-1)}}c_{\lfloor 2K\ln (2M) \rfloor +1}\big(\lfloor 2K\ln (2M) \rfloor +1\big)\nonumber\\
&\qquad\qquad+2\Big(\frac{\ln (2M)}{M}\Big)^{\frac{1}{2^n(n-1)}}\sum_{m=\lfloor 2K\ln (2M) \rfloor +2}^{m_2}c_m m^{-\frac{n}{2^n(n-1)}}K^{\frac{n}{2^n(n-1)}}\nonumber\\
&\leq \Big(\frac{\ln (2M)}{M}\Big)^{\frac{1}{2^n(n-1)}}\Big(1+2K^{\frac{n}{2^n(n-1)}}\sum_{m=\lfloor 2K\ln (2M) \rfloor +2}^{m_2}m^{-1-\frac{n}{2^n(n-1)}}\Big)\sup_{k\geq l}c_k k\nonumber\\
&\leq \Big(\frac{\ln (2M)}{M}\Big)^{\frac{1}{2^n(n-1)}}\bigg(1+2K^{\frac{n}{2^n(n-1)}}\Big(\frac{2^n(n-1)}{n}\Big)\big(\lfloor 2K\ln (2M) \rfloor\big)^{-\frac{n}{2^n(n-1)}}\bigg)\sup_{k\geq l}c_k k\nonumber\\
&\qquad\qquad\qquad\qquad\qquad\leq H\Big(\frac{\ln (2M)}{M}\Big)^{\frac{1}{2^n(n-1)}}\sup_{k\geq l}c_k k,
\end{align}
where $H>0$ depends only on $n$. Thus, from \eqref{sum_3_1}, \eqref{sum_3_2} and \eqref{sum_3_3} we have
\begin{align}\label{sum_3}
\sum_{m=\lfloor 2K\ln (2M) \rfloor +1}^{m_2}(c_m-c_{m+1})|S_m(x)|\leq H'\Big(\frac{\ln (2M)}{M}\Big)^{\frac{1}{2^n(n-1)}}\sup_{k\geq l}c_k k,
\end{align}
where $H'>0$ also depends only on $n$ and $f$.

For $K\leq m\leq 2K\ln (2M)$, according to \eqref{after_sqr},

\begin{align}\label{sum_2}
&\sum_{m=K}^{\lfloor 2K\ln (2M)\rfloor}(c_m-c_{m+1})|S_m(x)|\leq W\sum_{m=K}^{\lfloor 2K\ln (3M)\rfloor}(c_m-c_{m+1})\frac{m}{M^{\frac{1}{2^n(n-1)}}}\nonumber\\
&\leq \frac{W}{M^{\frac{1}{2^n(n-1)}}}\Big(c_KK+\sum_{m=K+1}^{\lfloor 2K\ln (3M)\rfloor}c_m\Big)\leq \frac{W}{M^{\frac{1}{2^n(n-1)}}}\Big(1+2\ln\frac{2K\ln (3M)}{K+1}\Big)\sup_{k\geq l}c_k k\nonumber\\
&\qquad\qquad\qquad\qquad\qquad\qquad\qquad\quad\qquad\leq 3W\frac{\ln\ln (3M)}{M^{\frac{1}{2^n(n-1)}}}\sup_{k\geq l}c_k k.
\end{align}
Combining estimates \eqref{sum_1}, \eqref{sum_3} and \eqref{sum_2}, we obtain
\begin{align}\label{m1m2}
&\sum_{m=m_1}^{m_2}(c_m-c_{m+1})|S_m(x)|\leq A\Big(M^{\frac{\varepsilon-1}{\varepsilon}}+\frac{1}{\ln (2M)}\Big)\sup_{k\geq l}c_k k \nonumber\\
&+\bigg(H'\Big(\frac{\ln (2M)}{M}\Big)^{\frac{1}{2^n(n-1)}}+3W\frac{\ln\ln (3M)}{M^{\frac{1}{2^n(n-1)}}}\bigg)\sup_{k\geq l}c_k k\leq A'\frac{\sup_{k\geq l}c_k k}{\ln (2M)},
\end{align}
where $A'>0$ depends only on $\varepsilon, n$ and $f$. So, since $M\in\mathfrak{M}$ and \eqref{M_i}, we have
\begin{align}\label{neudobn}
&\sum_{\underset{\text{inconvenient}}{m\geq l}}(c_m-c_{m+1})|S_m(x)|\leq\sum_{i\geq 1}A'\frac{\sup_{k\geq l}c_k k}{\ln (2M_i)}\leq A'\Big(\frac{1}{\ln 2}+\sum_{i\geq 2}\frac{1}{\ln M_i}\Big)\sup_{k\geq l}c_k k\nonumber\\
&\leq A'\Big(\frac{1}{\ln 2}+\frac{1}{\ln 2}\sum_{i\geq 2}\frac{1}{3^{i-2}}\Big)\sup_{k\geq l}c_k k\leq \frac{3}{\ln 2} A'\sup_{k\geq l}c_k k.
\end{align}

Thus, summing up \eqref{simple_sum}, \eqref{neudobn}, \eqref{udobn} and \eqref{taksebe}, we derive
\begin{align*}
\bigg|\sum_{m=l}^L c_m\sin f(m)x\bigg|\leq 2\sup_{k\geq l}c_k k+\frac{3}{\ln 2} A'\sup_{k\geq l}c_k k&+D''\sup_{k\geq l}c_k k+D_1''\sup_{k\geq l}c_k k\\
&\leq A''\sup_{k\geq l}c_k k,
\end{align*}
where $A''>0$ depends only on $\varepsilon, n$ and $f$. This means that the theorem is proved for $x\in \mathbb{R}\setminus\pi\mathbb{Q}$.

\vspace{2pt}
Finally, if $x\in \pi\mathbb{Q}$, we find $x'\notin \pi\mathbb{Q}$ such that $|x-x'|\leq L^{-n-1}c_1^{-1}\sup_{k\geq l}c_k k,$ then
\begin{align}\label{0end}
&\bigg|\sum_{m=l}^L c_m\sin f(m)x\bigg|\leq \bigg|\sum_{m=l}^L c_m\sin f(m)x'\bigg|+\bigg|\sum_{m=l}^L c_m\sin f(m)x-\sum_{m=l}^L c_m\sin f(m)x'\bigg|\nonumber\\
&\quad\leq A''\sup_{k\geq l}c_k k+C_f\sum_{m=l}^L c_m m^n|x-x'|\nonumber\\
&\quad\quad\leq A''\sup_{k\geq l}c_k k+C_fL^{-n-1}\sup_{k\geq l}c_k k\sum_{m=l}^L m^n\leq (A''+C_f)\sup_{k\geq l}c_k k,
\end{align}
which completes the proof.
\end{proof}

\section{The case of a power from (1,2)}

The feature of this case is that for $\alpha\in(1,2)$ the differences $(k+1)^{\alpha}-k^{\alpha}\;$ increase, and increase quite slowly. The idea of the proof is the following: to select 
blocks of such $k$ that the differences $(k+1)^{\alpha}x-k^{\alpha}x$, taken modulo $2\pi$, lie closely to $0$ or $2\pi$. Then the $``$steps$"$ between $k^{\alpha}x$ and $(k+1)^{\alpha}x$ in these blocks are small enough, and we can estimate the sums of the form $\sum_{k=k_1}^{k_1+s}\sin k^{\alpha}x$ in them by using Lemma \ref{lem_from_2}. For other $k$, a sum of the form $\sum_{k=k_1}^{k_1+s}\sin k^{\alpha}x$ slightly differs from the sum $$\sum_{k=0}^s\sin(x_0+k\gamma)=\frac{\cos\left(x_0-\frac{\gamma}{2}\right)-\cos\left(x_0+(2k+1)\frac{\gamma}{2}\right)}{2\sin\frac{\gamma}{2}},$$ where $\gamma$ is separated from $0$ and $2\pi$, hence, $\sin \frac{\gamma}{2}$ is separated from zero. The main difficulty lies in choosing the lengths of such blocks: they should not be too small to provide an appropriate estimate of the whole sum, but it should not be too large, otherwise the words $``$slightly differs$"$ will make no sence, since the differences $(k+1)^{\alpha}x-k^{\alpha}x$ will change a lot in a block.

\begin{proof}[Proof of Theorem \ref{thm} (c) for the case $\alpha\in(1,2)$] Let the condition $c_kk\to 0$ be satisfied. We will show that series \eqref{series} converges uniformly on the set $|x|\leq X<\infty$. Without loss of generality from now on we assume $x>0$ (the case $x=0$ is obvious). We fix some $\delta$ from the interval $(0,\frac{2-\alpha}{3})$. Let a natural number $l_0\geq 2$ be such that the following conditions are fulfilled:
\begin{align}\label{cond2}
\Big(\frac{\pi}{\alpha(\alpha-1)}-1\Big) l_0^2\geq \pi,\qquad l_0^{1-\frac{\alpha}{2}}>4\sqrt{\pi} \ln^2 l_0,\qquad l_0^{\frac{2-\alpha}{3}-\delta}>4 \ln^2 l_0.
\end{align}
Then for any $l\geq l_0$ all these conditions are satisfied as well. 

Consider
$$\sum_{k=l}^L c_k \sin k^{\alpha}x,$$
where $l\geq l_0$ and $x\in (0,X]$ is fixed. Let $m:=\lceil x^{-\frac{1}{\alpha}}\rceil$, so
\begin{align}\label{S1S2}
\bigg|\sum_{k=l}^L c_k \sin k^{\alpha}x\bigg|\leq \bigg|\sum_{k=l}^{m-1} c_k \sin k^{\alpha}x\bigg|+\bigg|\sum_{k=m}^L c_k \sin k^{\alpha}x\bigg|=:|S_1|+|S_2|.
\end{align}
If $m=1$, then $S_1=0$. Otherwise $2\leq m\leq x^{-\frac{1}{\alpha}}+1\leq 2x^{-\frac{1}{\alpha}}$, and we have
\begin{align}\label{S1}
|S_1|&\leq \sum_{k=l}^{m-1}c_k k^{\alpha} x\leq \sup_{k\geq l}c_k k\sum_{k=1}^{m-1}k^{\alpha-1}x\leq  x\sup_{k\geq l} c_k k \int\limits_1^{2x^{-\frac{1}{\alpha}}} y^{\alpha-1} dy\leq\frac{2^{\alpha}}{\alpha} \sup_{k\geq l} c_k k.
\end{align}
Further, denote
$$\Delta_k^1:=k^{\alpha}x-(k-1)^{\alpha}x,\qquad \Delta_k^2:=\Delta_k^1-\Delta_{k-1}^1,\qquad \tilde{\Delta}_k^1:=\Delta_k^1 \mod 2\pi,$$
so $\tilde{\Delta}_k^1\in [0,2\pi).$ Note that $\Delta_k^2$ decreases in $k$. Indeed, by Lagrange's theorem
\begin{align}\label{3}
\Big(\frac{\Delta_k^2}{x}\Big)_k'=\alpha\Big(k^{\alpha-1}&-2(k-1)^{\alpha-1}+(k-2)^{\alpha-1}\Big)\nonumber\\
&=\alpha(\alpha-1)\Big((k-1+\theta_1)^{\alpha-2}-(k-2+\theta_2)^{\alpha-2}\Big)<0,
\end{align}
where $\theta_1,\theta_2\in(0,1)$. Note also that
\begin{align}\label{4}
\Delta_k^2=\alpha x\Big((k-1+\theta_3)^{\alpha-1}-(k-2+\theta_4)^{\alpha-1}\Big)\leq 2\alpha(\alpha-1)x(k-2)^{\alpha-2},
\end{align}
where $\theta_3,\theta_4\in(0,1)$, and that
\begin{align}\label{5}
\Delta_k^2\geq \frac{1}{2}(\Delta_{k+1}^2+\Delta_k^2)&=\frac{1}{2}(\Delta_{k+1}^1-\Delta_{k-1}^1)=\frac{1}{2}\alpha x \Big((k+\theta_5)^{\alpha-1}-(k-2+\theta_6)^{\alpha-1}\Big)\nonumber\\
&\geq \frac{1}{2}\alpha(\alpha-1)x (k+1)^{\alpha-2},
\end{align}
where again $\theta_5,\theta_6\in(0,1)$. Let $$K_1:=\Big\{k:\tilde{\Delta}_{k+1}^1\in[0,m^{-\delta}]\cup[2\pi-m^{-\delta},2\pi]\Big\},\quad K_2:=\Big\{k:\tilde{\Delta}_{k+1}^1\in[m^{-\delta},2\pi-m^{-\delta}]\Big\}.$$ Then we have
\begin{align}\label{S2'S2''}
|S_2|\leq \bigg|\sum_{\underset{k\in K_1}{k=m}}^L c_k \sin k^{\alpha}x\bigg|+\bigg|\sum_{\underset{k\in K_2}{k=m}}^L c_k \sin k^{\alpha}x\bigg|=:|S_2'|+|S_2''|.
\end{align}

First we estimate $S_2'$. According to \eqref{3} and \eqref{4}, for $k\geq m+2$ there holds $\Delta_k^2\leq 2\alpha(\alpha-1)x m^{\alpha-2}$. Hence, we can find $p=p(m)$ such that $$p:=\min\Big\{p'>1:|\Delta_{m+p'}^1-\Delta_{m+1}^1-2\pi|\leq 2\alpha(\alpha-1)x m^{\alpha-2}\Big\}.$$ This yields
$$2\alpha (\alpha-1)xm^{\alpha-2}p\geq 2\pi -2\alpha(\alpha-1) x m^{\alpha-2},$$
so
\begin{align}\label{p}
p\geq -1+\frac{2\pi}{2\alpha(\alpha-1)x}m^{2-\alpha}\geq m^{2-\alpha}x^{-1},
\end{align}
due to the first condition from \eqref{cond2}.

Since $\Delta_k^1$ increases in $k$, (see, for instance, \eqref{5}), and since $p$ we chose as the minimal one, we have $0<\Delta_{m+p-1}^1-\Delta_{m+1}^1<2\pi$. Hence, among $\tilde{\Delta}_{m+1}^1,\;\tilde{\Delta}_{m+2}^1,...,\\\tilde{\Delta}_{m+p}^1$, there are not more than three blocks of consecutive $\tilde{\Delta}_i^1$, i.e., blocks from $\tilde{\Delta}_{i_1}^1, \tilde{\Delta}_{i_1+1}^1,...,\tilde{\Delta}_{i_1+i_2}^1$, with values increasing and lying in one of the intervals $[0,m^{-\delta}]$ or $[2\pi-m^{-\delta},2\pi]$. Focus on the case of $[0,m^{-\delta}]$, the second one is to treat similarly. Let our block be $\tilde{\Delta}_{s+1}^1,\;\tilde{\Delta}_{s+2}^1,...,\tilde{\Delta}_{s+v}^1$. Without loss of generality we can assume that $s^{\alpha}x\in [\pi u,\pi(u+1))=:I_u$ for some even $u$. Let $t$ be such that 
\begin{align*}
s^{\alpha}x+\sum_{i=0}^t \tilde{\Delta}_{s+i}^1\in I_u,\qquad s^{\alpha}x+\sum_{i=0}^{t+1} \tilde{\Delta}_{s+i}^1\notin I_u.
\end{align*}
Then the following inequality must be valid
\begin{align}\label{pokrmere}
\pi\geq (t-1)\Delta_{s+2}^2+(t-2)\Delta_{s+3}^2+...+1\cdot \Delta_{s+t}^2.
\end{align}
From \eqref{5} and \eqref{pokrmere} we have
\begin{align}\label{izpokrmere}
\pi\geq \frac{\alpha(\alpha-1)}{2}x \Big((t-1)(s+3)^{\alpha-2}+(t-2)(s+4)^{\alpha-2}+...+1\cdot (s+t+1)^{\alpha-2}\Big).
\end{align}
Note that the function $\kappa(y)=y(a-y)^{-c}+(b-y)(a-b+y)^{-c}$ does not increase in $c>0, \;a\geq b\geq 2y>0$. Indeed,
$$\kappa'(y)=(a-y)^{-c}-y(-c)(a-y)^{-c-1}-(a-b+y)^{-c}+(b-y)(-c)(a-b+y)^{-c-1}<0,$$
since $a-y\geq a-b+y$ and $y\leq b-y$. Therefore, for $c=2-\alpha,\;b=t,\;a=s+t+2$ and $y=t-i,\;i=1,2,...,\lfloor\frac{t-1}{2}\rfloor$, we have 
$$(t-i)(s+2+i)^{\alpha-2}+i(s+t+2-i)^{\alpha-2}\geq 2\frac{t}{2}\Big(s+2+\frac{t}{2}\Big)^{\alpha-2},$$
thus, using \eqref{izpokrmere},
\begin{align}\label{zamech}
\pi\geq \frac{\alpha(\alpha-1)}{2}x(t-1)&\frac{t}{2}\Big(s+2+\frac{t}{2}\Big)^{\alpha-2}\nonumber\\
&\geq \frac{\alpha(\alpha-1)}{2}x(t-1)\frac{t-1}{2}\Big(s+2+\frac{t-1}{2}\Big)^{\alpha-2}.
\end{align}
If $t-1\geq 2(s+2)$, then we have from \eqref{zamech}
$$\pi\geq \frac{\alpha(\alpha-1)}{2}x\frac{(t-1)^{\alpha}}{2}\geq \frac{\alpha(\alpha-1)}{4}(t-1)^{\alpha}m^{-\alpha},$$
from which $t-1\leq \big(\frac{4\pi}{\alpha(\alpha-1)}\big)^{\frac{1}{\alpha}} m\leq \frac{4\pi}{\alpha-1} m \leq \frac{4\pi}{\alpha-1} (s+2)$. Thus, we get from \eqref{zamech}
$$\pi\geq \frac{\alpha(\alpha-1)}{2}x\frac{(t-1)^2}{2}\Big(\frac{4\pi}{\alpha-1}+1\Big)^{\alpha-2}(s+2)^{\alpha-2}>\frac{(\alpha-1)^2}{20\pi}(t-1)^{2}(s+2)^{\alpha-2}x,$$
hence,
\begin{align}\label{est_t}
t<\frac{8}{\alpha-1}(s+2)^{1-\frac{\alpha}{2}}x^{-\frac{1}{2}}+1\leq \frac{27}{\alpha-1}s^{1-\frac{\alpha}{2}}x^{-\frac{1}{2}}.
\end{align}
Let $t_0:=t$, and we define $t_i$ for $i\geq 1$ in the following way: $$s^{\alpha}x+\sum_{j=1}^{t_i} \tilde{\Delta}_{s+j}^1\in I_{u+i},\quad s^{\alpha}x+\sum_{j=1}^{t_i+1} \tilde{\Delta}_{s+j}^1\notin I_{u+i}.$$ We also denote by $R$ the minimal even number for which there holds $s^{\alpha}x+\sum_{j=1}^{v} \tilde{\Delta}_{s+j}^1<\pi(u+R+2)$. Then by an argument similar to that of \eqref{est_t} we get
\begin{align}\label{est_t_1}
t_1<\frac{54}{\alpha-1}s^{1-\frac{\alpha}{2}}x^{-\frac{1}{2}}.
\end{align}
We also have
\begin{align}\label{sum_R}
\sum_{k=s}^{s+v}c_k\sin k^{\alpha}x =\sum_{k=s}^{s+t_0}c_k\sin k^{\alpha}x+\sum_{i=0}^{\frac{R}{2}-1}\sum_{k=s+t_{2i}+1}^{s+t_{2i+2}}c_k\sin k^{\alpha}x+\sum_{k=s+t_R}^{s+v}c_k\sin k^{\alpha}x.
\end{align}
Note that due to \eqref{est_t}
\begin{align}\label{uchest_t}
\sum_{k=s}^{s+t_0}c_k\sin k^{\alpha}x\leq t c_s<\frac{27}{\alpha-1}s^{1-\frac{\alpha}{2}}x^{-\frac{1}{2}} c_s\leq \frac{27}{\alpha-1}s^{-\frac{\alpha}{2}}x^{-\frac{1}{2}}\sup_{k\geq l} c_k k.
\end{align}

\begin{lemma}\label{lem_from_2} Let the points $y_1,...,y_k$ be such that $0 < y_1\leq y_2-y_1\leq y_3-y_2\leq ...\leq y_k-y_{k-1},\; y_k\leq  2\pi$ and let the number $q$ be such that $y_q\leq \pi<y_{q+1}$ and $y_{q+1}-y_q=a<\pi$. Then $\sum_{i=1}^k\sin y_i\geq -\sin\frac{a}{2}$.
\end{lemma}

\begin{proof} Let $\mu$ be such that $\sin y_{\mu}\geq \sin y_i$ for all $i$, and $\nu$ be such that $\sin y_{\nu}\leq \sin y_i$ for all $i$. Note that then 
\begin{align*}
y_1,...,y_{\mu-1}\in\left[0,\frac{\pi}{2}\right],&\;y_{\mu+1},...,y_q\in\left[\frac{\pi}{2},\pi\right],\\
&\;y_{q+1},...,y_{\nu-1}\in\left[\pi,\frac{3\pi}{2}\right],\;y_{\nu+1},...,y_k\in\left[\frac{3\pi}{2},\pi\right].
\end{align*}
In this case, for $1\leq i\leq \mu-1$ we have $y_{i+1}-y_i\leq a$, hence,
$$\sin y_i\geq \max\Big\{\sin\big(y_{\mu}-a(M-i)\big),\;0\Big\}.$$
Thus,
$$\sum_{i=1}^{\mu-1}\sin y_i\geq \sum_{j=1}^{\left\lceil \frac{y_{\mu}}{a}\right\rceil -1}\sin(y_{\mu}-aj).$$
Similarly, since for $\mu+1\leq i\leq q$ also $y_{i+1}-y_i\leq a$,
$$\sum_{i=\mu+1}^{q}\sin y_i\geq \sum_{j=1}^{\left\lceil \frac{\pi-y_{\mu}}{a}\right\rceil -1}\sin(y_{\mu}+aj).$$
Further, since for $q+1\leq i\leq k-1$ there holds $y_{i+1}-y_i\geq a$, we have
$$\sum_{i=q+1}^{\nu-1}\sin y_i\geq \sum_{j=1}^{\left\lceil \frac{y_{\nu}-\pi}{a}\right\rceil -1}\sin(y_{\nu}-aj)$$
and
$$\sum_{i=\nu+1}^{k}\sin y_i\geq \sum_{j=1}^{\left\lceil \frac{2\pi-y_{\nu}}{a}\right\rceil -1}\sin(y_{\nu}+aj).$$
Then
\begin{align*}
\sum_{i=1}^{q}\sin y_i &\geq \sum_{j=-\left\lceil \frac{y_{\mu}}{a}\right\rceil+1}^{\left\lceil \frac{\pi-y_{\mu}}{a}\right\rceil -1}\sin(y_{\mu}+aj)\\
&=\frac{\cos\left(y_{\mu}-\lceil \frac{y_{\mu}}{a}\rceil a+\frac{a}{2}\right)-\cos\big(y_{\mu}+\lceil \frac{\pi-y_{\mu}}{a}\rceil a-\frac{a}{2}\big)}{2\sin\frac{a}{2}}\geq \frac{\cos\frac{a}{2}}{\sin\frac{a}{2}}.
\end{align*}
At the same time
\begin{align*}
\sum_{i=q+1}^{k}\sin y_i &\geq \sum_{j=-\left\lceil \frac{y_{\nu}-\pi}{a}\right\rceil+1}^{\left\lceil \frac{2\pi-y_{\nu}}{a}\right\rceil -1}\sin(y_{\nu}+aj)\\
&=\frac{\cos\left(y_{\nu}-\lceil \frac{y_{\nu}-\pi}{a}\rceil a+\frac{a}{2}\right)-\cos\left(y_{\nu}+\lceil \frac{2\pi-y_{\nu}}{a}\rceil a-\frac{a}{2}\right)}{2\sin\frac{a}{2}}\geq \frac{-1}{\sin\frac{a}{2}}.
\end{align*}
So,
$$\sum_{i=1}^{k}\sin y_i\geq\frac{\cos\frac{a}{2}-1}{\sin\frac{a}{2}}\geq \frac{\cos^2\frac{a}{2}-1}{\sin\frac{a}{2}}=-\sin\frac{a}{2}.$$
\end{proof}

\begin{corollary}\label{cor_2} Let the points $y_1,...,y_k$ be as in Lemma \ref{lem_from_2} and the sequence $\{a_j\}$ be nonincreasing. Then $\sum_{i=1}^k a_i\sin y_i\geq -a_{q+1}\sin\frac{a}{2}$.
\end{corollary}

\begin{proof} We have
\begin{align*}
\sum_{i=1}^{k}a_i\sin y_i=\sum_{i=1}^{q}a_i\sin y_i&+\sum_{i=q+1}^{k}a_i\sin y_i\\
&\geq a_{q+1}\sum_{i=1}^{q}\sin y_i+a_{q+1}\sum_{i=q+1}^{k}\sin y_i\geq -a_{q+1}\sin\frac{a}{2}
\end{align*}
by Lemma \ref{lem_from_2}.
\end{proof}
\vspace{4 pt}

By Corollary \ref{cor_2} we have for any $0\leq i\leq \frac{R}{2}-1$
\begin{align}\label{t_i}
\sum_{k=s+t_{2i}+1}^{s+t_{2i+2}}c_k\sin k^{\alpha}x\leq \frac{m^{-\delta}}{2}c_{s+t_{2i+1}+1},
\end{align}
\begin{align}\label{t_i_R}
\sum_{k=s+t_R}^{s+v}c_k\sin k^{\alpha}x\leq \frac{m^{-\delta}}{2}c_{s+t_{R+1}+1}.
\end{align}
Thus, from \eqref{sum_R} and \eqref{uchest_t}, \eqref{t_i} and \eqref{t_i_R}, we derive
\begin{align}\label{poluch}
\sum_{k=s}^{s+v}c_k\sin k^{\alpha}x\leq \Big(\frac{R}{2}+1\Big)\frac{m^{-\delta}}{2}c_s+\frac{27}{\alpha-1}s^{-\frac{\alpha}{2}}x^{-\frac{1}{2}}\sup_{k\geq l} c_k k.
\end{align}
Note that $v$ must fulfill the inequality
\begin{align}\label{v}
\Delta_{s+v}^1-\Delta_{s+1}^1\leq m^{-\delta},
\end{align}
and by Lagrange's theorem, for some $\theta_7,\;\theta_8\in(0,1)$, the left hand side of \eqref{v} can be written as follows:
$$\alpha x\Big((s+v-1+\theta_7)^{\alpha-1}-(s+\theta_8)^{\alpha-1}\Big)\geq \alpha x \Big((s+v-1)^{\alpha-1}-(s+1)^{\alpha-1}\Big),$$
which yields
$$(s+v-1)^{\alpha-1}\leq m^{-\delta}x^{-1}+(s+1)^{\alpha-1},$$
hence,
\begin{align}\label{do_bern}
\Big(1+\frac{v-2}{s+1}\Big)^{\alpha-1}\leq m^{-\delta}x^{-1}(s+1)^{1-\alpha}+1.
\end{align}
By Bernoulli's inequality, the left hand side of \eqref{do_bern} is not less than $1+(\alpha-1)\frac{v-2}{s+1},$ so we get
$$(\alpha-1)\frac{v-2}{s+1}\leq m^{-\delta}x^{-1}(s+1)^{1-\alpha},$$
then
\begin{align}\label{v_new}
v\leq \frac{1}{(\alpha-1)}m^{-\delta}x^{-1}(s+1)^{2-\alpha}+2\leq \frac{1+2X(\alpha-1)}{(\alpha-1)}m^{-\delta}x^{-1}(s+1)^{2-\alpha}.
\end{align}
Thus, it follows from \eqref{3}, \eqref{4} and \eqref{v_new} that
\begin{align}\label{R}
R\leq \frac{1}{2\pi}\Delta_{s+2}^2 v &\leq \frac{\alpha(\alpha-1)}{\pi}s^{\alpha-2}x \frac{1+2X(\alpha-1)}{(\alpha-1)}m^{-\delta}x^{-1}(s+1)^{2-\alpha}\nonumber\\
&\leq\frac{2\cdot2^{2-\alpha}\cdot (1+2X)}{\pi}m^{-\delta}\leq (2+4X)m^{-\delta}.
\end{align}
From \eqref{poluch} and \eqref{R} we derive
\begin{align*}
\sum_{k=s}^{s+v}c_k\sin k^{\alpha}x\leq (2+4X)m^{-\delta}&\frac{m^{-\delta}}{2}c_s+\frac{27}{\alpha-1}s^{-\frac{\alpha}{2}}x^{-\frac{1}{2}}\sup_{k\geq l}c_k k\\
&\leq  \Big((1+2X)m^{-1-2\delta}+\frac{27}{\alpha-1}s^{-\frac{\alpha}{2}}x^{-\frac{1}{2}}\Big)\sup_{k\geq l} c_k k,
\end{align*}
hence,
$$\sum_{k=m}^{m+p}c_k\sin k^{\alpha}x\leq 3 \Big((1+2X)m^{-1-2\delta}+\frac{27}{\alpha-1}m^{-\frac{\alpha}{2}}x^{-\frac{1}{2}}\Big)\sup_{k\geq l} c_k k.$$
So, in view of \eqref{p},
\begin{align}\label{S2'}
S_2'\leq 3&\sum_{i=0}^{\infty}\Big((1+2X)w_i^{-1-2\delta}+\frac{27}{\alpha-1}w_i^{-\frac{\alpha}{2}}x^{-\frac{1}{2}}\Big)\sup_{k\geq l} c_k k \nonumber\\
\leq &3 \Big((1+2X)\frac{m^{-2\delta}}{2\delta}+\frac{27}{\alpha-1}\sum_{i=0}^{\infty}w_i^{-\frac{\alpha}{2}}x^{-\frac{1}{2}}\Big)\sup_{k\geq l} c_k k,
\end{align}
where $w_0:=m$ and $w_{i+1}:=w_i+w_i^{2-\alpha}x^{-1}\geq w_i+1$ for $i\geq 0$, therefore, 
\begin{align}\label{w_i}
w_i\underset{i\to\infty}{\to}\infty.
\end{align}
Recall that $m\geq l\geq l_0\geq 2$ and consider the function
$$F(m):=\int\limits_m^{\infty}\frac{dy}{y\ln ^2 y}=\frac{1}{\ln m}.$$
According to \eqref{w_i}, 
\begin{align}\label{sogl}
F(m)=\sum_{j=0}^{\infty}\int\limits_{w_j}^{w_{j+1}}\frac{dy}{y\ln^2 y}=:\sum_{j=0}^{\infty}W_j.
\end{align}
Suppose that, for $j=0,...,J$ and for these values only, there holds $x^{-1}>w_j^{\alpha-1}$, then for $j=0,...,J-1$ we have $w_{j+1}\geq 2w_j$, hence,
\begin{align}\label{sum_J}
\sum_{i=0}^{J}w_i^{-\frac{\alpha}{2}}x^{-\frac{1}{2}}\leq m^{-\frac{\alpha}{2}}x^{-\frac{1}{2}}\sum_{i=0}^{\infty}2^{-\frac{i\alpha}{2}}\leq \frac{1}{1-2^{-\frac{\alpha}{2}}}\leq 4.
\end{align}
Besides, for $j>J$, there holds $x^{-1}\leq w_j^{\alpha-1}$, and then, using the inequality $\ln(1+y)\geq y/2$ which is true for $y\leq 1$, we obtain
\begin{align}\label{W_j}
W_j=\frac{1}{\ln w_j}-&\frac{1}{\ln \left(w_j+w_j^{2-\alpha}x^{-1}\right)}=\frac{\ln \left(1+w_j^{1-\alpha}x^{-1}\right)}{\ln w_j \ln\left(w_j+w_j^{2-\alpha}x^{-1}\right)}\nonumber\\
&\geq \frac{w_j^{1-\alpha}x^{-1}}{2\ln w_j \ln(2w_j)}\geq w_j^{-\frac{\alpha}{2}}x^{-\frac{1}{2}}.
\end{align}
Here we used the double inequality
$$w_j^{1-\frac{\alpha}{2}}x^{-\frac{1}{2}}\geq w_j^{1-\frac{\alpha}{2}}\pi^{-\frac{1}{2}}\geq 4\ln^2 w_j,$$
which is valid since $w_j\geq m\geq l_0$ by the second condition from \eqref{cond2}. Thus, from \eqref{sogl} and \eqref{W_j},
\begin{align}\label{ln_2}
\sum_{i=J+1}^{\infty}w_i^{-\frac{\alpha}{2}}x^{-\frac{1}{2}}\leq \sum_{i=J+1}^{\infty}W_j\leq F(m)\leq \frac{1}{\ln 2}.
\end{align}
Combining estimates \eqref{sum_J} and \eqref{ln_2}, we derive from \eqref{S2'} that
\begin{align}\label{S2'new}
S_2'\leq 3 \left(\frac{1+2X}{2\delta}+\frac{27}{\alpha-1}\Big(4+\frac{1}{\ln 2}\Big)\right)\sup_{k\geq l} c_k k.
\end{align}
Replacing \eqref{sum_R} by
$$\sum_{k=s}^{s+v}c_k\sin k^{\alpha}x =\sum_{k=s}^{s+t_1}c_k\sin k^{\alpha}x+\sum_{i=2}^{\frac{R}{2}-1}\sum_{k=s+t_{2i-1}+1}^{s+t_{2i+1}}c_k\sin k^{\alpha}x+\sum_{k=s+t_R}^{s+v}c_k\sin k^{\alpha}x$$
and using the same argument, we get, with the help of \eqref{est_t_1},
\begin{align}\label{min_S2'}
S_2'\geq -3 \left(\frac{1+2X}{2\delta}+\frac{54}{\alpha-1}\Big(4+\frac{1}{\ln 2}\Big)\right)\sup_{k\geq l} c_k k.
\end{align}
Summing up \eqref{S2'new} and \eqref{min_S2'}, we have finally
\begin{align}\label{|S2'|}
|S_2'|\leq C(\alpha, X)\sup_{k\geq l} c_k k.
\end{align}

Consider now $S_2''$. Let $m'=m'(m)\geq m$ be the first number such that $m'\in K_2$. Put $Q=Q(m):=\lceil m^{\frac{2-\alpha}{3}}\rceil$. Note that for $k\in K_2$ there holds 
\begin{align}\label{Delta_k+1}
\frac{m^{-\delta}}{2}\leq \frac{\tilde{\Delta}_{k+1}^1}{2}\leq \pi-\frac{m^{-\delta}}{2}.
\end{align}
Applying the Abel transformation, we get
\begin{align}\label{abel}
\sum_{k=m'}^{m'+Q-1}c_k \sin k^{\alpha} x=\sum_{q=0}^{Q-1}(c_{m'+q}-c_{m'+q+1})\sum_{k=m'}^{m'+q}\sin k^{\alpha} x+c_{m'+Q}\sum_{k=m'}^{m'+Q-1}\sin k^{\alpha}x.
\end{align}
Besides,
$$(m'+q)^{\alpha} x\underset{\mod 2\pi}{=} (m')^{\alpha} x+\sum_{t=1}^q \tilde{\Delta}_{m'+t}^1,$$
and then from \eqref{3} and \eqref{4}
\begin{align}\label{34}
\big|(m'+t)^{\alpha}-(m')^{\alpha}-\tilde{\Delta}_{m+1}^1 t\big|\leq \frac{t(t-1)}{2}\Delta_{m'+2}^2\leq t(t-1)\alpha(\alpha-1) x m^{\alpha-2}.
\end{align}
Since for arbitrary $g,h\in\mathbb{R}$ there holds $|\sin(g+h)-\sin g|\leq |h|,$ it follows from \eqref{34} that for $q\leq Q-1$
\begin{align}\label{sum_of_squares}
&\bigg|\sum_{k=m'}^{m'+q}\sin k^{\alpha}x-\sum_{t=0}^{q}\sin\left((m')^{\alpha}x+\tilde{\Delta}_{m+1}^1 t\right)\bigg|\leq \frac{Q(Q+1)(2Q+1)}{6}\alpha(\alpha-1)x m^{\alpha-2}\nonumber\\
&\leq Q^3\alpha(\alpha-1)x m^{\alpha-2}\leq (2m^{\frac{2-\alpha}{3}})^3\alpha(\alpha-1)x m^{\alpha-2}=2^{2-\alpha}\alpha(\alpha-1)x\leq 4X.
\end{align}
Besides, taking into account \eqref{Delta_k+1},
\begin{align}\label{pimdelta}
\bigg|\sum_{t=0}^q\sin \left((m')^{\alpha}x+\tilde{\Delta}_{m+1}^1 t\right)\bigg|&=\Bigg|\frac{\cos \left((m')^{\alpha}x-\frac{\tilde{\Delta}_{m+1}^1}{2}\right)-\cos \left((m')^{\alpha}x+\frac{\tilde{\Delta}_{m+1}^1(2q+1)}{2}\right)}{2\sin\frac{\tilde{\Delta}_{m+1}^1}{2}}\Bigg|\nonumber\\
&\leq \frac{2}{2\frac{2}{\pi}\frac{m^{-\delta}}{2}}=\pi m^{\delta}.
\end{align}
From \eqref{sum_of_squares} and \eqref{pimdelta} we have
\begin{align}\label{3pimdelta}
\bigg|\sum_{k=m'}^{m'+q}\sin k^{\alpha}x\bigg|\leq \pi m^{\delta}+4X\leq (\pi+4X) m^{\delta},
\end{align}
and from \eqref{3pimdelta} and \eqref{abel} it follows that
\begin{align}\label{last}
\bigg|\sum_{k=m'}^{m'+Q-1}c_k\sin k^{\alpha}x\bigg|\leq c_{m'} (\pi+4X) m^{\delta} \leq c_{m} (\pi+4X) m^{\delta}\leq (\pi+4X) m^{\delta-1}\sup_{k\geq l} c_k k.
\end{align}

Let $Q'=Q'(m)\geq Q(m)$ be the minimal number such that $m'+Q'\in K_2$. Denote $m_0:=m,\;m_{i+1}:=m'(m_i)+Q'(m_i)$ for any $i\geq 0$. Since
$$Q'\geq Q \geq m^{\frac{2-\alpha}{3}},$$
we have
\begin{align}\label{takkakto}
m_{i+1}\geq m_i+m_i^{\frac{2-\alpha}{3}}.
\end{align}
Notice that in the sum on the left hand side of \eqref{last} there can appear blocks of such $k$ that $k\in K_1$ and the values $\tilde{\Delta}_{k+1}^1$ in a block increase and belong to an interval $[0,m^{-\delta}]$ or $[2\pi-m^{-\delta},2\pi]$. The sum over each one of these blocks can be estimated as in \eqref{poluch}, where we estimated the corresponding block of $S_2'$. So, from \eqref{|S2'|}, \eqref{last} and \eqref{takkakto}, and also recalling that $\delta<\frac{2-\alpha}{3}<1$, we get
\begin{align}\label{|S2''|}
|S_2''|&\leq C(\alpha, X)\sup_{k\geq l}c_k k+(\pi+4X) \sup_{k\geq l} c_k k\sum_{i=0}^{\infty}c_{m_i}m_i^{\delta-1}\nonumber\\
&\leq C(\alpha, X)\sup_{k\geq l}c_k k+(\pi+4X) \sup_{k\geq l} c_k k\sum_{i=0}^{\infty}z_i^{\delta-1},
\end{align}
where $z_0:=m,\;z_{i+1}:=z_i+z_i^{\frac{2-\alpha}{3}}\geq z_i+1$ for any $i$, and hence, $z_i\underset{i\to\infty}{\to}\infty.$ Therefore, 
\begin{align}\label{Final}
F(m)=\sum_{j=0}^{\infty}\int\limits_{z_j}^{z_{j+1}}\frac{dy}{y\ln ^2 y}=:\sum_{j=0}^{\infty} Z_j.
\end{align}
For the sake of convinience denote $\frac{2-\alpha}{3}=:\gamma>\delta$. Using the inequality $\ln(1+y)\geq y/2$, valid for $y\leq 1$, we have
\begin{align}\label{Z_j}
Z_j=\frac{1}{\ln z_j}-\frac{1}{\ln \left(z_j+z_j^{\gamma}\right)}=\frac{\ln \big(1+z_j^{\gamma-1}\big)}{\ln z_j \ln\left(z_j+z_j^{\gamma}\right)}\geq \frac{z_j^{\gamma-1}}{2\ln z_j \ln(2z_j)}>z_j^{\delta-1}.
\end{align}
The latter inequality in \eqref{Z_j} is due to the inequality
$$z_j^{\gamma-\delta}>4\ln^2 z_j,$$
which is true since $z_j\geq l_0$ and in view of the third condition from \eqref{cond2}. Thus, \eqref{|S2''|}, \eqref{Final} and \eqref{Z_j} imply
\begin{align}\label{theend}
|S_2''|&\leq C(\alpha, X)\sup_{k\geq l} c_k k+(\pi+4x) \sup_{k\geq l} c_k k\sum_{i=0}^{\infty}Z_i=C(\alpha, X)\sup_{k\geq l} c_k k\nonumber\\
&+(\pi+4X) F(m)\sup_{k\geq l} c_k k\leq \Big(C(\alpha, X)+\frac{\pi+4X}{\ln 2}\Big)\sup_{k\geq n} c_k k.
\end{align}
Finally, combining \eqref{S1S2}, \eqref{S1}, \eqref{S2'S2''}, \eqref{|S2'|} and \eqref{theend}, we get
$$\bigg|\sum_{k=l}^L c_k \sin k^{\alpha}x\bigg|\leq \Big(\frac{2^{\alpha}}{\alpha}+2C(\alpha, X)+\frac{\pi+4X}{\ln 2}\Big)\sup_{k\geq l} c_k k,$$
which assures that in the case of fulfilling the condition $c_kk\to 0$ our series converges uniformly.
\end{proof}

\section{The case of a power from (0,1)}
\begin{proof}[Proof of Theorem \ref{thm} (c) for the case $\alpha\in(0,1)$]
Suppose that the condition $c_kk\to 0$ is satisfied. We will show that the series \eqref{series} converges uniformly on the set $|x|\leq X<\infty$. Without loss of generality from now on we assume $x>0$. Take an odd number $D\geq 3$ fulfilling the following conditions
\begin{align}\label{cond}
(\pi X^{-1})^{\frac{1}{\alpha}}D^{\frac{1}{\alpha}-1}&\geq 12\alpha,\quad\Big(1+\frac{1}{D}\Big)^{\frac{1}{\alpha}-1}\leq \frac{4}{3},\nonumber\\
&\Big(1-\frac{3}{2\alpha}\frac{1}{D}\Big)^{\alpha-1}\leq \frac{4}{3},\quad \Big(1+\frac{3}{2D}\Big)^{\frac{1}{\alpha}-2}\leq 2,
\end{align}
and let $E:=D+1$. Consider the sum $\sum_{k=l}^{L}c_k\sin k^{\alpha}x$ at an arbitrary point $x\in(0,X]$. 
If $x\leq \pi L^{-\alpha},$ then
\begin{align}\label{case_1}
0\leq \sum_{k=l}^{L}c_k\sin k^{\alpha}x\leq x\sum_{k=l}^L c_k k^{\alpha}\leq x\sup_{k\geq l} c_k k\sum_{k=l}^L k^{\alpha-1}&\leq \pi L^{-\alpha}\frac{(2L)^{\alpha}}{\alpha}\sup_{k\geq l} c_k k\nonumber\\
&=:C_1\sup_{k\geq l} c_k k.
\end{align}

If $x\geq \pi l^{-\alpha}$ and $L^{\alpha}x-l^{\alpha}x\leq 6\pi,$ we have $L^{\alpha}-l^{\alpha}\leq \frac{6\pi}{x}\leq 6 l ^{\alpha},$ hence, $L\leq 7^{\frac{1}{\alpha}}l$, therefore,
\begin{align}\label{case_2}
\bigg|\sum_{k=l}^L c_k\sin k^{\alpha}x\bigg|\leq \sum_{k=l}^L c_k \leq c_l(L-l+1)<7^{\frac{1}{\alpha}}l c_l
&\leq 7^{\frac{1}{\alpha}}\sup_{k\geq l} c_k k\nonumber\\
&=:C_2(\alpha)\sup_{k\geq l} c_k k.
\end{align}

The remaining case is that of $x\geq \pi l^{-\alpha}$ and $L^{\alpha}x-l^{\alpha}x > 6\pi.$

Let odd numbers $d_1,\;d_2$ and even numbers $e_1,\;e_2$ be such that
\begin{align*}
\pi(e_1-2)<xl^{\alpha}&\leq \pi e_1,\quad \pi(d_1-2)<xl^{\alpha}\leq \pi d_1,\\
&\quad \pi e_2\leq xL^{\alpha}< \pi(e_2+2),\quad \pi d_2\leq xL^{\alpha}\leq \pi(d_2+2).
\end{align*}

Note that for any $\gamma>0$ and $d\geq 3$ there holds
\begin{align*}
&F(\gamma,d)=F(\gamma,d,\alpha):=\frac{\lfloor(\gamma d)^{\frac{1}{\alpha}}\rfloor-\left\lfloor\big(\gamma (d-2)\big)^{\frac{1}{\alpha}}\right\rfloor}{\left\lfloor\big(\gamma (d-2)\big)^{\frac{1}{\alpha}}\right\rfloor+1}\leq \frac{2\Big((\gamma d)^{\frac{1}{\alpha}}-\big(\gamma (d-2)\big)^{\frac{1}{\alpha}}\Big)}{\big(\gamma (d-2)\big)^{\frac{1}{\alpha}}}\nonumber\\
&\leq \frac{2\Big((\gamma d)^{\frac{1}{\alpha}}-\big(\gamma (d-2)\big)^{\frac{1}{\alpha}}\Big)}{\big(\gamma (d-2)\big)^{\frac{1}{\alpha}}}=2\Big(\Big(\frac{d}{d-2}\Big)^{\frac{1}{\alpha}}-1\Big)\leq 2\big (3^{\frac{1}{\alpha}}-1\big)=:C.
\end{align*}

Thus, we have
\begin{align}\label{cherezC1}
\bigg|\sum_{k=l}^{\lfloor (\pi x^{-1}d_1)^{\frac{1}{\alpha}}\rfloor}c_k\sin k^{\alpha}x\bigg|&\leq c_l \sum_{k=\lfloor (\pi x^{-1}(d_1-2))^{\frac{1}{\alpha}}\rfloor+1}^{\lfloor (\pi x^{-1} d_1)^{\frac{1}{\alpha}}\rfloor}1\nonumber\\
\leq c_l\Big(\left\lfloor \big(\pi x^{-1}(d_1-2)\big)^{\frac{1}{\alpha}}\right\rfloor+1\Big)F(\pi x^{-1},d_1)&\leq lc_l F(\pi x^{-1},d_1)\leq C\sup_{k\geq l} c_k k.
\end{align}
Similarly,
\begin{align}\label{cherezC2}
\bigg|\sum_{k=l}^{\lfloor (\pi x^{-1} e_1)^{\frac{1}{\alpha}}\rfloor}c_k\sin k^{\alpha}x\bigg|\leq C\sup_{k\geq l} c_k k.
\end{align}
Further,
\begin{align}\label{cherezC3}
&\bigg|\sum_{k=\lfloor (\pi x^{-1} e_2)^{\frac{1}{\alpha}}\rfloor+1}^{L}c_k\sin k^{\alpha}x\bigg|\leq c_{\lfloor (\pi x^{-1} e_2)^{\frac{1}{\alpha}}\rfloor+1} \sum_{k=\lfloor (\pi x^{-1} e_2)^{\frac{1}{\alpha}}\rfloor+1}^{\lfloor (\pi x^{-1} (e_2+2))^{\frac{1}{\alpha}}\rfloor}1\nonumber\\
&\qquad\leq c_{\lfloor (\pi x^{-1} e_2)^{\frac{1}{\alpha}}\rfloor+1}\Big(\lfloor (\pi x^{-1} e_2)^{\frac{1}{\alpha}}\rfloor+1\Big)F(\pi x^{-1},e_2+2)\leq C \sup_{k\geq l} c_k k.
\end{align}
Similarly,
\begin{align}\label{cherezC4} 
\bigg|\sum_{k=\lfloor (\pi x^{-1} d_2)^{\frac{1}{\alpha}}\rfloor+1}^{L}c_k\sin k^{\alpha}x\bigg|\leq C \sup_{k\geq l} c_k k.
\end{align}

Now consider the sum
\begin{align*}
&S(d):=\sum_{k=\lfloor (\pi x^{-1} d)^{\frac{1}{\alpha}}\rfloor+1}^{\lfloor (\pi x^{-1} (d+2))^{\frac{1}{\alpha}}\rfloor}\sin k^{\alpha}x=\sum_{k=\lfloor (\pi x^{-1} d)^{\frac{1}{\alpha}}\rfloor+1}^{\lfloor (\pi x^{-1} (d+\frac{1}{2}))^{\frac{1}{\alpha}}\rfloor}+\sum_{k=\lfloor (\pi x^{-1} (d+\frac{1}{2}))^{\frac{1}{\alpha}}\rfloor+1}^{\lfloor (\pi x^{-1} (d+1))^{\frac{1}{\alpha}}\rfloor}\\
&+\sum_{k=\lfloor (\pi x^{-1} (d+1))^{\frac{1}{\alpha}}\rfloor+1}^{\lfloor (\pi x^{-1} (d+\frac{3}{2}))^{\frac{1}{\alpha}}\rfloor}+\sum_{k=\lfloor (\pi x^{-1} (d+\frac{3}{2})^{\frac{1}{\alpha}}\rfloor+1}^{\lfloor (\pi x^{-1} (d+2))^{\frac{1}{\alpha}}\rfloor}\sin k^{\alpha}x=:S_1(d)+S_2(d)+S_3(d)+S_4(d),
\end{align*}
where $d\geq D$ is an odd number. 

First we show that the sum $S_2(d)+S_3(d)$ cannot be too large, because most of the summands contained in the sums $S_2(d)$ and $S_3(d)$ can be split into pairs so that the sum of any pair would be close to zero and nonpositive. Let in the $s$-th pair the values $k$ be $\lfloor (\pi x^{-1} (d+1))^{\frac{1}{\alpha}}\rfloor+s$ and $\lfloor (\pi x^{-1} (d+1))^{\frac{1}{\alpha}}\rfloor-1-s$, where
\begin{align}\label{s}
s=0,1,...,&\min\Bigg\{\left\lfloor \Big(\pi x^{-1} \Big(d+\frac{3}{2}\Big)\Big)^{\frac{1}{\alpha}}\right\rfloor-\left\lfloor \Big(\pi x^{-1} (d+1)\Big)^{\frac{1}{\alpha}}\right\rfloor,\nonumber\\
&\left\lfloor \Big(\pi x^{-1} (d+1)\Big)^{\frac{1}{\alpha}}\right\rfloor-\left\lfloor \Big(\pi x^{-1} \Big(d+\frac{1}{2}\Big)\Big)^{\frac{1}{\alpha}}\right\rfloor-1\Bigg\}.
\end{align}
Note that there is exactly one pair consisting of summands of $S_2(d)$ and any other pair consists of a summand of $S_2(d)$ and a summand of $S_3(d)$. The sum of the values of the $s$-th pair is
\begin{align}\label{sumpair}
&\sin \bigg(\left\lfloor \Big(\pi x^{-1} (d+1)\Big)^{\frac{1}{\alpha}}\right\rfloor+s\bigg)^{\alpha}x+\sin \bigg(\left\lfloor \Big(\pi x^{-1} (d+1)\Big)^{\frac{1}{\alpha}}\right\rfloor-1-s\bigg)^{\alpha}x\qquad\nonumber\\
&\;=2\sin \Bigg(\bigg(\left\lfloor \Big(\pi x^{-1} (d+1)\Big)^{\frac{1}{\alpha}}\right\rfloor+s\bigg)^{\alpha}+\bigg(\left\lfloor \Big(\pi x^{-1} (d+1)\Big)^{\frac{1}{\alpha}}\right\rfloor-1-s\bigg)^{\alpha}\Bigg)\frac{x}{2}\nonumber\\
&\;\;\;\;\cdot\cos \Bigg(\bigg(\left\lfloor \Big(\pi x^{-1} (d+1)\Big)^{\frac{1}{\alpha}}\right\rfloor+s\bigg)^{\alpha}-\bigg(\left\lfloor \Big(\pi x^{-1} (d+1)\Big)^{\frac{1}{\alpha}}\right\rfloor-1-s\bigg)^{\alpha}\Bigg)\frac{x}{2}.
\end{align}
According to \eqref{s}, the argument of any cosine in \eqref{sumpair} lies on the interval $[0,\pi/2]$, hence, all the cosines are nonnegative. Let us show now that the arguments of all sines in \eqref{sumpair} lie in the half-interval $[\pi d,\pi(d+1))$ which would lead to nonpositivity of these sines. Due to convexity of the function $\chi(y)=y^{\alpha}$ ($\chi''(y)=\alpha(\alpha-1)y^{\alpha-2}<0$ for $y>0$) on $\mathbb{R}^+$, the argument of the sine does not exceed
$$2\left(\left\lfloor \Big(\pi x^{-1} (d+1)\Big)^{\frac{1}{\alpha}}\right\rfloor-\frac{1}{2}\right)^{\alpha}\frac{x}{2}<\pi(d+1).$$
At the same time, there holds
\begin{align}\label{dop}
\Big(\pi x^{-1} \Big(d+\frac{3}{2}\Big)\Big)^{\frac{1}{\alpha}}&-\Big(\pi x^{-1} (d+1)\Big)^{\frac{1}{\alpha}}+1\nonumber\\
&<\Big(\pi x^{-1} (d+2)\Big)^{\frac{1}{\alpha}}-\Big(\pi x^{-1} (d+1)\Big)^{\frac{1}{\alpha}}-\frac{1}{2},
\end{align}
since by Lagrange's theorem there exists $\theta\in(0,\frac{1}{2})$ such that
\begin{align*}
(\pi x^{-1})^{\frac{1}{\alpha}}\Big((d+2)^{\frac{1}{\alpha}}-\Big(d+\frac{3}{2}\Big)^{\frac{1}{\alpha}}\Big)&\geq (\pi x^{-1})^{\frac{1}{\alpha}}\frac{1}{2\alpha}\Big(d+\frac{3}{2}+\theta\Big)^{\frac{1}{\alpha}-1}\\
\geq (\pi x^{-1})^{\frac{1}{\alpha}}\frac{1}{2\alpha}\Big(d+\frac{3}{2}+\theta\Big)^{\frac{1}{\alpha}-1}&\geq (\pi x^{-1})^{\frac{1}{\alpha}}\frac{1}{2\alpha}D^{\frac{1}{\alpha}-1}\geq 6>\frac{3}{2}
\end{align*}
by the first condition of \eqref{cond}. Hence, from \eqref{dop}, 
$$s\leq \Big(\pi x^{-1} (d+2)\Big)^{\frac{1}{\alpha}}-\Big(\pi x^{-1} (d+1)\Big)^{\frac{1}{\alpha}}-\frac{1}{2},$$
therefore,
\begin{align}\label{10}
\frac{s+\frac{1}{2}}{\big(\pi x^{-1} (d+1)\big)^{\frac{1}{\alpha}}}<\Big(\frac{d+2}{d+1}\Big)^{\frac{1}{\alpha}}-1\leq \frac{4}{3\alpha(d+1)},
\end{align}
since the function $t(y)=(1+y)^{\frac{1}{\alpha}}-1-4y/3\alpha$ vanishes for $y=0$ and $t'(y)=((1+y)^{\frac{1}{\alpha}-1}-4/3)/\alpha\leq 0$ for $y\leq 1/D$ due to the second condition of \eqref{cond}. Also, taking into account the first condition of \eqref{cond} and the fact that $d\geq D$,
\begin{align}\label{11}
\frac{3}{\big(2\pi x^{-1} (d+1)\big)^{\frac{1}{\alpha}}}<\frac{2}{(\pi x^{-1})^{\frac{1}{\alpha}}(d+1)^{\frac{1}{\alpha}}}\leq \frac{1}{6\alpha (d+1)}.
\end{align}
Thus, by \eqref{10} and \eqref{11}, the argument of any sine at the right hand side of \eqref{sumpair} is not less than
\begin{align}\label{argsin}
&\bigg(\Big(\pi x^{-1} (d+1)\Big)^{\frac{1}{\alpha}}-\frac{3}{2}+\max \Big(s+\frac{1}{2}\Big)\bigg)^{\alpha}\frac{x}{2}\nonumber\\
&\quad+\bigg(\Big(\pi x^{-1} (d+1)\Big)^{\frac{1}{\alpha}}-\frac{3}{2}-\max \Big(s+\frac{1}{2}\Big)\bigg)^{\alpha}\frac{x}{2}\nonumber\\
&\qquad\geq \Big(1-\frac{3}{2(\pi x^{-1}(d+1))^{1/\alpha}}+\frac{4}{3\alpha(d+1)}\Big)^{\alpha}\frac{\pi (d+1)}{2}\nonumber\\
&\qquad\qquad+\Big(1-\frac{3}{2(\pi x^{-1}(d+1))^{1/\alpha}}-\frac{4}{3\alpha(d+1)}\Big)^{\alpha}\frac{\pi (d+1)}{2}\nonumber\\
&\geq \bigg(\Big(1-\frac{1}{6\alpha(d+1)}+\frac{4}{3\alpha(d+1)}\Big)^{\alpha}+\Big(1-\frac{1}{6\alpha(d+1)}-\frac{4}{3\alpha(d+1)}\Big)^{\alpha}\bigg)\frac{\pi (d+1)}{2}\nonumber\\
&\qquad\qquad\qquad\qquad\qquad\geq\bigg(1+\Big(1-\frac{3}{2\alpha(d+1)}\Big)^{\alpha}\bigg)\frac{\pi (d+1)}{2}.
\end{align}
We will show that the latter expression is not less than $\pi d$. It is sufficient to prove that the function
$$g(y)=1+\Big(1-\frac{3}{2\alpha}y\Big)^{\alpha}-2+2y=\Big(1-\frac{3}{2\alpha}y\Big)^{\alpha}-1+2y$$
is not negative at the point $y=(d+1)^{-1}$. Note that $g(0)=0$ and
$$g'(y)=-\frac{3}{2}\Big(1-\frac{3}{2\alpha}y\Big)^{\alpha-1}+2\geq 0$$
for $y\leq 1/D$ by the third of the conditions \eqref{cond}. Thus, by \eqref{argsin} and the observation above it follows that the argument of any sine at the right hand side of \eqref{sumpair} is not less than $\pi d$. Besides, it is easy to see that any of these arguments is also not greater than $(\pi+1)d$, hence, all the sines at the right hand side of \eqref{sumpair} are nonpositive, and this implies nonpositivity of the whole sum of the chosen pairs. If there is a summand of $S_2(d)$ not belonging to any pair, we bound it above by zero. 

Let us estimate the number of the summands of $S_3(d)$ which could be left without a pair. If there exist such summands, then since we have exactly one pair consisting of summands of $S_2(d)$ and all other pairs consists of a summand of $S_2(d)$ and a summand of $S_3(d)$, therefore, the number of summands of $S_3(d)$ left without a pair is exactly
\begin{align}\label{bezpary}
&\left(\left\lfloor \Big(\pi x^{-1} \Big(d+\frac{3}{2}\Big)\Big)^{\frac{1}{\alpha}}\right\rfloor-\left\lfloor \Big(\pi x^{-1} (d+1)\Big)^{\frac{1}{\alpha}}\right\rfloor\right)\nonumber\\
&\qquad\qquad-\left(\left\lfloor \Big(\pi x^{-1} (d+1)\Big)^{\frac{1}{\alpha}}\right\rfloor-\left\lfloor \Big(\pi x^{-1} \Big(d+\frac{1}{2}\Big)\Big)^{\frac{1}{\alpha}}\right\rfloor-2\right)\nonumber\\
&\leq\Big(\pi x^{-1} \Big(d+\frac{3}{2}\Big)\Big)^{\frac{1}{\alpha}}-2\Big(\pi x^{-1} (d+1)\Big)^{\frac{1}{\alpha}}+ \Big(\pi x^{-1} \Big(d+\frac{1}{2}\Big)\Big)^{\frac{1}{\alpha}}+4\nonumber\\
&\qquad\qquad\qquad\qquad\qquad\leq \frac{2}{\alpha}\Big(\frac{1}{\alpha}-1\Big)d^{\frac{1}{\alpha}-2}(\pi x^{-1})^{\frac{1}{\alpha}}+4.
\end{align}
Here we used Lagrange's theorem for the function $w(y)=y^{\frac{1}{\alpha}}$:
\begin{align*}
w(y+1)-2w\Big(y+\frac{1}{2}\Big)+w(y)&=w'\Big(y+\frac{1}{2}+\theta_1\Big)-w'(y+\theta_2)\\
&=\Big(\frac{1}{2}+\theta_1-\theta_2\Big)w''(y+\theta_0),
\end{align*}
where $\theta_1,\;\theta_2\in[0,\frac{1}{2}],\;\theta_0\in[0,1]$. Therefore, 
\begin{align*}
&w(y+1)-2w\Big(y+\frac{1}{2}\Big)+w(y)\leq \sup_{[y+\frac{1}{2},y+\frac{3}{2}]}w''(z)\\
&\qquad=\frac{1}{\alpha}\Big(\frac{1}{\alpha}-1\Big)\max\bigg\{\Big(d+\frac{1}{2}\Big)^{\frac{1}{\alpha}-2},\Big(d+\frac{3}{2}\Big)^{\frac{1}{\alpha}-2}\bigg\}\leq \frac{2}{\alpha}\Big(\frac{1}{\alpha}-1\Big) d^{1/\alpha-2}
\end{align*}
according to the fourth condition of \eqref{cond}. Thus, estimate \eqref{bezpary} is valid.

From the argument above it follows that
\begin{align}\label{S2S3}
S_2(d)+S_3(d)\leq \frac{2}{\alpha}\Big(\frac{1}{\alpha}-1\Big)d^{\frac{1}{\alpha}-2}(\pi x^{-1})^{\frac{1}{\alpha}}+4.
\end{align}

Let us show now that the sum $S_1(d)$ is slightly different from $S_2(d)$, and $S_4(d)$ --- from $S_3(d)$. We will construct a one-to-one correspondence between
\begin{align}\label{snew}
s=1,2,...,\left\lfloor \Big(\pi x^{-1} (d+1)\Big)^{\frac{1}{\alpha}}\right\rfloor-\left\lfloor \Big(\pi x^{-1} \Big(d+\frac{1}{2}\Big)\Big)^{\frac{1}{\alpha}}\right\rfloor=:s_{max}
\end{align}
and some of
\begin{align}\label{k_s}
k_s=\left\lfloor \Big(\pi x^{-1} \Big(d+\frac{1}{2}\Big)\Big)^{\frac{1}{\alpha}}\right\rfloor-\lfloor (\pi x^{-1} d)^{\frac{1}{\alpha}}\rfloor+1,...,\left\lfloor \Big(\pi x^{-1} (d+1)\Big)^{1/\alpha}\right\rfloor-\lfloor (\pi x^{-1} d)^{\frac{1}{\alpha}}\rfloor,
\end{align}
so that
\begin{align}\label{k_s_cond}
k_s:=\min\Big\{k\in\mathbb{N}:\Big(\lfloor (\pi x^{-1} d)^{\frac{1}{\alpha}}\rfloor+k\Big)^{\alpha}\geq  \pi x^{-1}(2d+1)-\Big(\lfloor (\pi x^{-1} d)^{\frac{1}{\alpha}}\rfloor+s\Big)^{\alpha}\Big\}.
\end{align}
Then, using
$$\pi d\leq \Big(\lfloor (\pi x^{-1} d)^{\frac{1}{\alpha}}\rfloor+s\Big)^{\alpha}x\leq \pi\Big(d+\frac{1}{2}\Big)\leq \Big(\lfloor (\pi x^{-1} d)^{\frac{1}{\alpha}}\rfloor+k_s\Big)^{\alpha}x\leq \pi(d+1)$$
and
$$\pi\Big(d+\frac{1}{2}\Big)-\Big(\lfloor (\pi x^{-1} d)^{\frac{1}{\alpha}}\rfloor+s\Big)^{\alpha}x\leq \Big(\lfloor (\pi x^{-1} d)^{\frac{1}{\alpha}}\rfloor+k_s\Big)^{\alpha}x-\pi\Big(d+\frac{1}{2}\Big),$$
we get
\begin{align}\label{18}
\sin \Big(\lfloor (\pi x^{-1} d)^{\frac{1}{\alpha}}\rfloor+s\Big)^{\alpha}x\leq \sin\Big(\lfloor (\pi x^{-1} d)^{\frac{1}{\alpha}}\rfloor+k_s\Big)^{\alpha}x.
\end{align}
Note that
\begin{align}\label{zametim}
\Big(\pi x^{-1}(2d+1)-\Big((\pi x^{-1} d)^{\frac{1}{\alpha}}-1+s\Big)^{\alpha}\Big)^{\frac{1}{\alpha}}&-(\pi x^{-1} d)^{\frac{1}{\alpha}}+2\nonumber\\
&\leq \Big(\pi x^{-1} (d+1)\Big)^{\frac{1}{\alpha}}-1,
\end{align}
since
$$\Big(\pi x^{-1}(2d+1)-\pi x^{-1} d\Big)^{\frac{1}{\alpha}}-(\pi x^{-1} d)^{\frac{1}{\alpha}}+3-\Big(\pi x^{-1} (d+1)\Big)^{\frac{1}{\alpha}}=3- (\pi x^{-1} d)^{\frac{1}{\alpha}}\leq 0$$
due to $d\geq D\geq 3$. Therefore, it follows from \eqref{zametim} that
$$\Big(\pi x^{-1}(2d+1)-\Big(\lfloor (\pi x^{-1} d)^{\frac{1}{\alpha}}\rfloor+s\Big)^{\alpha}\Big)^{\frac{1}{\alpha}}-\lfloor (\pi x^{-1} d)^{\frac{1}{\alpha}}\rfloor+1\leq \left\lfloor \Big(\pi x^{-1} (d+1)\Big)^{\frac{1}{\alpha}}\right\rfloor,$$
hence, for any $s$ of \eqref{snew} there exists $k_s$ satisfying \eqref{k_s} and \eqref{k_s_cond}.

We will also show that $k_{s_1}\neq k_{s_2}$ for $s_1\neq s_2$. Since $k_s$ does not increase when $s$ increases, it suffices to show that $k_s> k_{s+1}$. Indeed, we can see from \eqref{k_s_cond} that
$$k_s\geq \Big(\pi x^{-1}(2d+1)-\Big(\lfloor (\pi x^{-1} d)^{\frac{1}{\alpha}}\rfloor+s\Big)^{\alpha}\Big)^{\frac{1}{\alpha}}-\lfloor (\pi x^{-1} d)^{\frac{1}{\alpha}}\rfloor>k_s-1,$$
so, it is sufficient to prove validity of
\begin{align}\label{sprav}
\Big(\pi x^{-1}(2d+1)-\Big(\lfloor (\pi x^{-1} d)^{\frac{1}{\alpha}}\rfloor+s\Big)^{\alpha}\Big)^{\frac{1}{\alpha}}&-\lfloor (\pi x^{-1} d)^{\frac{1}{\alpha}}\rfloor \nonumber\\
 > \Big(\pi x^{-1}(2d+1)-\Big(\lfloor (\pi x^{-1} d)^{\frac{1}{\alpha}}\rfloor+s+1\Big)^{\alpha}\Big)^{\frac{1}{\alpha}}&-\lfloor (\pi x^{-1} d)^{\frac{1}{\alpha}}\rfloor+1.
\end{align}
For the sake of brevity we denote $a:=\pi x^{-1}(2d+1),\;b:=\lfloor(\pi x^{-1} d)^{\frac{1}{\alpha}}\rfloor$ and consider the function
$$h_{a,b}(s)=\Big(a-(b+s)^{\alpha}\Big)^{\frac{1}{\alpha}}.$$
Then by Lagrange's theorem
$h_{a,b}(s)-h_{a,b}(s+1)=-h'_{a,b}(s_0),$
where $s_0\in(1,s_{max})$. Besides,
$$h'_{a,b}(s)=-\Big(a-(b+s)^{\alpha}\Big)^{\frac{1}{\alpha}-1}(b+s)^{\alpha-1},$$
i.e., $|h'_{a,b}|$ decreases in $b+s$, and hence, using that $b+s\leq \left(\pi x^{-1}\left(d+\frac{1}{2}\right)\right)^{\frac{1}{\alpha}}$ according to \eqref{snew}, we have on the interval $(1,s_{max})$
$$|h'_{a,b}(s)|> \Big(\pi x^{-1}\Big(d+\frac{1}{2}\Big)\Big)^{\frac{1}{\alpha}-1}\Big(\Big(\pi x^{-1}\Big(d+\frac{1}{2}\Big)\Big)^{\frac{1}{\alpha}}\Big)^{\alpha-1}=1.$$
Thus, $h_{a,b}(s)-h_{a,b}(s+1)>1$, which implies validity of \eqref{sprav}.

So, each $s$ satisfying \eqref{snew} corresponds injectively to $k_s$ satisfying \eqref{k_s} and \eqref{k_s_cond}, so that for each $s$ there holds \eqref{18}, i.e., any summand of $S_1(d)$ is bounded above by the corresponding summand of $S_2(d)$. The number of the summands of $S_2(d)$ which are not used in this estimate is
\begin{align*}
&\left\lfloor \Big(\pi x^{-1} (d+1)\Big)^{\frac{1}{\alpha}}\right\rfloor-\left\lfloor \Big(\pi x^{-1} \Big(d+\frac{1}{2}\Big)\Big)^{\frac{1}{\alpha}}\right\rfloor\\
&\quad\;\;\quad\quad\quad\quad\quad\quad\quad -\bigg(\left\lfloor \Big(\pi x^{-1} \Big(d+\frac{1}{2}\Big)\Big)^{\frac{1}{\alpha}}\right\rfloor-\lfloor (\pi x^{-1} d)^{\frac{1}{\alpha}}\rfloor\bigg)\\
&\leq \Big(\pi x^{-1} (d+1)\Big)^{\frac{1}{\alpha}}+(\pi x^{-1} d)^{\frac{1}{\alpha}}-2\Big(\pi x^{-1} \Big(d+\frac{1}{2}\Big)\Big)^{\frac{1}{\alpha}}+2\\
&\qquad\qquad\qquad\qquad\qquad\qquad\qquad\qquad\leq \frac{2}{\alpha}\Big(\frac{1}{\alpha}-1\Big)d^{\frac{1}{\alpha}-2}(\pi x^{-1})^{\frac{1}{\alpha}}+2,
\end{align*}
similarly as \eqref{bezpary}. Thus,
\begin{align}\label{to}
S_1(d)\leq S_2(d)+\frac{2}{\alpha}\Big(\frac{1}{\alpha}-1\Big)d^{\frac{1}{\alpha}-2}(\pi x^{-1})^{\frac{1}{\alpha}}+2.
\end{align}

Proceeding with the same argument for $S_3(d)$ and $S_4(d)$, we get for $d\geq D$
\begin{align}\label{S4}
S_4(d)\leq S_3(d)+\frac{2}{\alpha}\Big(\frac{1}{\alpha}-1\Big)&(d+1)^{\frac{1}{\alpha}-2}(\pi x^{-1})^{\frac{1}{\alpha}}+2\nonumber\\
&\leq S_3(d)+\frac{4}{\alpha}\Big(\frac{1}{\alpha}-1\Big)d^{\frac{1}{\alpha}-2}(\pi x^{-1})^{\frac{1}{\alpha}}+2
\end{align}
due to the fourth condition of \eqref{cond}.

Finally, summing up \eqref{cherezC1}, \eqref{cherezC4}, \eqref{S2S3}, \eqref{to} and \eqref{S4}, we derive
\begin{align}\label{gather}
&\sum_{k=l}^L c_k\sin k^{\alpha}x\leq 2C\sup_{k\geq l} c_k k+\sum_{\underset{d \;\text{is odd}}{d\geq d_1}}^{D-2}\sum_{k=\lfloor (\pi x^{-1} d)^{\frac{1}{\alpha}}\rfloor+1}^{\lfloor (\pi x^{-1} (d+2))^{\frac{1}{\alpha}}\rfloor} c_k\sin k^{\alpha}x\nonumber\\
&+\sum_{\underset{d \;\text{is odd}}{d\geq D}}^{d_2-2}\sum_{k=\lfloor (\pi x^{-1} d)^{\frac{1}{\alpha}}\rfloor+1}^{\lfloor (\pi x^{-1} (d+2))^{\frac{1}{\alpha}}\rfloor} c_k\sin k^{\alpha}x\leq 2C\sup_{k\geq l} c_k k\nonumber\\
&+c_{\lfloor (\pi x^{-1} d_1)^{\frac{1}{\alpha}}\rfloor+1}\Big(\lfloor (\pi x^{-1} D)^{\frac{1}{\alpha}}\rfloor-\lfloor (\pi x^{-1} d_1)^{\frac{1}{\alpha}}\rfloor\Big)+\sum_{\underset{d \;\text{is odd}}{d\geq D}}^{d_2-2}c_{\lfloor (\pi x^{-1} (d+1))^{\frac{1}{\alpha}}\rfloor} S(d)\nonumber\\
&\leq 2C\sup_{k\geq l} c_k k+2\bigg(\Big(\frac{D}{d_1}\Big)^{\frac{1}{\alpha}}-1\bigg)\sup_{k\geq l} c_k k\nonumber\\
&+2\sum_{\underset{d \;\text{is odd}}{d\geq D}}^{d_2-2}c_{\lfloor (\pi x^{-1} (d+1))^{\frac{1}{\alpha}}\rfloor} \bigg(S_2(d)+S_3(d)+\frac{3}{\alpha}\Big(\frac{1}{\alpha}-1\Big)d^{\frac{1}{\alpha}-2}(\pi x^{-1})^{\frac{1}{\alpha}}+2\bigg)\nonumber\\
&\leq 2C\sup_{k\geq l} c_k k+2(D^{\frac{1}{\alpha}}-1)\sup_{k\geq l} c_k k+2\sum_{\underset{d \;\text{is odd}}{d\geq D}}^{d_2-2} c_{\lfloor (\pi x^{-1} (d+1))^{\frac{1}{\alpha}}\rfloor}\nonumber\\
&\qquad\qquad\cdot\bigg(\frac{2}{\alpha}\Big(\frac{1}{\alpha}-1\Big)d^{\frac{1}{\alpha}-2}(\pi x^{-1})^{\frac{1}{\alpha}}+4+\frac{3}{\alpha}\Big(\frac{1}{\alpha}-1\Big)d^{\frac{1}{\alpha}-2}(\pi x^{-1})^{\frac{1}{\alpha}}+2\bigg)\nonumber\\
&\leq 2C\sup_{k\geq l} c_k k+2(D^{\frac{1}{\alpha}}-1)\sup_{k\geq l} c_k k\nonumber\\
&\qquad\qquad\qquad\quad+2\sup_{k\geq l} c_k k\sum_{d\geq d_1} \frac{5}{\alpha}\Big(\frac{1}{\alpha}-1\Big)d^{-2}+\frac{6}{(\pi x^{-1})^{\frac{1}{\alpha}}}d^{-\frac{1}{\alpha}}\nonumber\\
&\leq \left(2C+2(D^{\frac{1}{\alpha}}-1)+\frac{10}{\alpha}\Big(\frac{1}{\alpha}-1\Big)D^{-1}+\frac{6}{(\pi X^{-1})^{\frac{1}{\alpha}}}\Big(\frac{1}{\alpha}-1\Big) D^{1-\frac{1}{\alpha}}\right)\sup_{k\geq l} c_k k\nonumber\\
&\leq  \left(2C+2(D^{\frac{1}{\alpha}}-1)+\Big(\frac{10}{\alpha}+2X^2\Big)\Big(\frac{1}{\alpha}-1\Big)\right)\sup_{k\geq l} c_k k=:C_3\sup_{k\geq l} c_k k.
\end{align}

Similarly, by the same argument for $e_1,\;e_2$ and $E$ in place of $d_1,\;d_2$ and $D$ and bounding $S(e)$ below, using \eqref{cherezC2} and \eqref{cherezC3}, we get
\begin{align}\label{analog}
\sum_{k=l}^L c_k\sin k^{\alpha}x\geq -C_4\sup_{k\geq l} c_k k.
\end{align}

Gathering \eqref{case_1}, \eqref{case_2}, \eqref{gather} and \eqref{analog}, we have
$$\bigg|\sum_{k=l}^L c_k\sin k^{\alpha}x\bigg|\leq \max\{C_1,\;C_2,\;C_3,\;C_4\}\sup_{k\geq l} c_k k,$$
which completes the proof of uniform convergence.
\end{proof}

\section{Proof of Theorem 2}

\begin{proof}[Proof of Theorem \ref{thm2}] The part (a) of Theorem \ref{thm2} follows clearly from Theorem \ref{thm} (a).

In view of Theorem \ref{thm} (b), (c), it suffices for the proof of the corresponding parts of Theorem \ref{thm2} to show that for any $\alpha>0$ the condition $c_kk\to 0$ is necessary and sufficient for uniform convergence of series \eqref{series} on a set containing for some $\gamma\geq 2$ a discrete $(\alpha,\gamma)$-neighbourhood of zero. Suppose that series \eqref{series} converges uniformly on some set $X$ containing a discrete $(\alpha,\gamma)$-neighbourhood of zero and let $2\leq\gamma$ and $N$ be the numbers from the definition of such a neighbourhood. Take an arbitrary $\varepsilon>0$. Then there exists $l_0=l_0(\varepsilon)\in\mathbb{N},\;l_0\geq N,$ such that for any $L>l\geq l_0$ and any $x\in X$ there holds $\Big|\sum_{k=l}^{L}c_k\sin k^{\alpha} x\Big|<\varepsilon.$ So, taking any $l\ge l_0$ and putting $x_0=\frac{\pi}{{\gamma}^{\alpha+1}l^{\alpha}}$ (either $x_0$ or $-x_0$ contains in $X$), we obtain
$$\varepsilon>\bigg|\sum_{k=l+1}^{2l}c_k\sin k^{\alpha} x_0\bigg|=\bigg|\sum_{k=l+1}^{2l}c_k\sin k^{\alpha} \frac{\pi}{\gamma^{\alpha+1}l^{\alpha}}\bigg|.$$
Note that the argument of any sine here does not exceed $\frac{\pi}{2}$, hence,
\begin{align}\label{021}
\varepsilon>\frac{2}{\pi}\sum_{k=l+1}^{2l}c_kk^{\alpha} \frac{\pi}{{\gamma}^{\alpha+1}l^{\alpha}}\geq 2{\gamma}^{-\alpha-1}\sum_{k=l+1}^{2l} c_k\geq 2{\gamma}^{-\alpha-1}lc_{2l}={\gamma}^{-\alpha-1}c_{2l}2l,
\end{align}
i.e., $c_{2l}2l\leq {\gamma}^{\alpha+1}\varepsilon$. Besides,
\begin{align}\label{022}
c_{2l+1}(2l+1)\leq c_{2l}4l\leq 2{\gamma}^{\alpha+1}\varepsilon,
\end{align}
which assures necessity of the condition.  
\end{proof}

\begin{proof}[Proof of Remark \ref{rbvs_rem}] Estimates of the proof of Theorem \ref{thm} (a), (b) remain true up to constants if we replace the differences $c_m-c_{m+1}$ by their absolute values. Indeed, it follows from the relations
\begin{align}\label{01}
\sum_{k=l}^L|c_k-c_{k+1}|k^{\xi}&=l^{\xi}\sum_{k=l}^L|c_k-c_{k+1}|+\sum_{k=l}^L \big((k+1)^{\xi}-k^{\xi}\big)\sum_{j=l}^L|c_k-c_{k+1}|\nonumber\\
&\leq V c_l l^{\xi}+VC(\xi)\sum_{k=l}^Lc_k k^{\xi-1},
\end{align}
where $\xi>0,\;V$ is from \eqref{rbvs}, and
\begin{align}\label{02}
c_k\leq c_m+\sum_{l=m}^{k-1}|c_l-c_{l+1}|\leq (V+1)c_m
\end{align}
for $k>m$. Inequality \eqref{01} implies validity of \eqref{udobn}, \eqref{taksebe}, \eqref{sum_1_1}, \eqref{sum_1_2}, \eqref{sum_3_2}, \eqref{sum_3_3} and \eqref{sum_2} with appropriate modifications, while the inequality \eqref{02} --- validity \eqref{00}, \eqref{0end}, \eqref{021} and \eqref{022}.
\end{proof}

{\bf Aknowledgements.} The author is deeply grateful to M. Dyachenko and S. Tikhonov for constant discussion of the results and constructive advice on the work presentation and also to referees for useful remarks contributed to the quality of the paper.

The research was carried out with the support of the Foundation for the Advancement of Theoretical Physics and Mathematics $``$BASIS$"$ 19-8-2-28-1.


\end{document}